\documentclass[10pt]{siamltex}
\usepackage{amsmath,amsfonts,amscd,amssymb,bm,cite}
\usepackage{mathrsfs,listings,url}
\usepackage{graphics,color,ulem}
\usepackage{float}
\usepackage[titletoc,page]{appendix}
\usepackage{subfigure}
\usepackage[colorlinks=true, bookmarksopen,
pdfauthor={CST},
pdfcreator={pdftex},
pdfsubject={algorithms},
linkcolor={blue},
anchorcolor={black},
citecolor={firebrick},
filecolor={magenta},
menucolor={black},
pagecolor={red},
plainpages=false,pdfpagelabels,
urlcolor={db} ]{hyperref}
\usepackage[capitalise]{cleveref}
\usepackage{comment}
\usepackage{verbatim}
\usepackage{marginnote}
\usepackage{float}
\usepackage[makeroom]{cancel}
\usepackage[pdftex]{graphicx}
\usepackage[section]{placeins}
\usepackage[dvipsnames]{xcolor}
\usepackage{mathptmx}
\usepackage{cases}
\usepackage{subeqnarray}
\usepackage{hhline}
\usepackage[linesnumbered,ruled,vlined]{algorithm2e}
\usepackage{xcolor}
\usepackage{tikz}
\usepackage{caption} 
\usepackage{lmodern}
\usepackage{array} 

\usepackage{comment}
\newcommand\Tstrut{\rule{0pt}{2.6ex}} 

\usetikzlibrary{arrows}
\usetikzlibrary{positioning,shapes,snakes,calc,decorations,decorations.markings}

\newtheorem{remark}{Remark}
\includecomment{confidential}
\excludecomment{confidential}
\restylefloat{table}
\allowdisplaybreaks
\definecolor{db}{rgb}{0.0470,0,0.5294}
\definecolor{dg}{rgb}{0.0,0.392,0.0}
\definecolor{firebrick}{rgb}{0.698,0.133,0.133}
\definecolor{bl}{rgb}{0.0,0.0,0.0}
\definecolor{linen}{rgb}{0.980,0.941,0.902}
\definecolor{ivory}{rgb}{1.0,1.0,0.941}
\definecolor{aliceblue}{rgb}{0.941,0.973,1.0}
\definecolor{beige}{rgb}{0.961,0.961,0.863}
\definecolor{tan}{rgb}{0.824,0.706,0.549}
\definecolor{lightsteelblue}{rgb}{0.690,0.769,0.871}
\definecolor{paleturquoise}{rgb}{0.686,0.933,0.933}
\definecolor{lightblue}{rgb}{0.678,0.847,0.902}
\definecolor{skyblue}{rgb}{0.529,0.808,0.922}
\definecolor{palegoldenrod}{rgb}{0.933,0.910,0.667}
\definecolor{lightgoldenrod}{rgb}{0.933,0.867,0.510}
\definecolor{lightyellow}{rgb}{1.0,1.0,0.878}
\definecolor{yellow}{rgb}{1.0,1.0,0.0}
\definecolor{lightyellow1}{rgb}{1.0,1.0,0.878}
\definecolor{lemonchiffon}{rgb}{1.0,0.980,0.804}
\definecolor{myyellow}{rgb}{1,1,.9}
\definecolor{darkgreen}{rgb}{0.0,0.392,0.0}
\definecolor{darkviolet}{rgb}{0.580,0.0,0.827}
\definecolor{lightsalmon}{rgb}{1.0,0.627,0.478}
\definecolor{orange}{rgb}{1.0,0.647,0.0}
\definecolor{darkblue}{rgb}{0.00,0.00,0.55}
\numberwithin{equation}{section}

\crefname{table}{table}{tables}
\Crefname{table}{Table}{Tables}
\crefname{figure}{figure}{figures}
\Crefname{figure}{Figure}{Figures}

\begin{document}
	
	\title{The Efficient Variable Time-stepping DLN Algorithms for the Allen-Cahn Model} 
	\author{
            Yiming Chen\thanks{
			Department of Mathematics, The Ohio State University, Columbus, OH 43210,
			USA. Email: \href{mailto:chen.11042@osu.edu}{chen.11042@osu.edu}. }
            \and
            Dianlun Luo\thanks{
			Department of Mathematics, Columbia University, New York, NY 10027,
			USA. Email: \href{mailto:dl3572@columbia.edu}{dl3572@columbia.edu}. }
            \and 
			Wenlong Pei\thanks{
			Department of Mathematics, The Ohio State University, Columbus, OH 43210,
			USA. Email: \href{mailto:pei.176@osu.edu}{pei.176@osu.edu}. } 
			\and    
			Yulong Xing\thanks{
			Department of Mathematics, The Ohio State University, Columbus, OH 43210,
			USA. Email: \href{mailto:xing.205@osu.edu}{xing.205@osu.edu}. } 
		    }
	\date{\emty}
	\maketitle
	
	\begin{abstract}
		We consider a family of variable time-stepping Dahlquist-Liniger-Nevanlinna (DLN) schemes, which is unconditional non-linear stable and second order accurate, for the Allen-Cahn equation. The finite element methods are used for the spatial discretization. For the non-linear term, we combine the DLN scheme with two efficient temporal algorithms: partially implicit modified algorithm and scalar auxiliary variable algorithm. For both approaches, we prove the unconditional, long-term stability of the model energy under any arbitrary time step sequence. 
        Moreover, we provide rigorous error analysis for the partially implicit modified algorithm with variable time-stepping.
        Efficient time adaptive algorithms based on these schemes are also proposed. Several one- and two-dimensional numerical tests are presented to verify the properties of the proposed time adaptive DLN methods.

	\end{abstract}
	
	
	
	
	
	
	
	
	\section{Introduction}
	The Allen-Cahn equation was first introduced by Allen and Cahn in \cite{AC79_ActaMetall} to describe the motion of anti-phase boundaries in crystalline solids. Given the domain $\Omega \subset \mathbb{R}^{d}$ ($d = 1,2,3$), $u(x,t)$ models the difference between the concentrations of two mixtures' components at time $t$, and is governed by
    \begin{gather}
        \begin{cases}
        \displaystyle u_{t} - \epsilon^{2} \Delta u + \frac{d}{du} F(u) = 0, \qquad & x \in \Omega, \ 0 < t \leq T, \\
        \displaystyle \frac{\partial u}{\partial \vec{n}} = 0,  &\text{on } \partial{\Omega},
        \end{cases}
    \label{eq:AllenCahn_Eq}
    \end{gather}
    where $\epsilon$ is the interfacial parameter. The Helmholtz free-energy density $F(u)$ and its derivative take the form
    \begin{gather*}
        F(u) = \frac{1}{4} (u^2 - 1)^{2}, \qquad
        f(u) = \frac{d}{du} F(u) = u^{3} - u.
    \end{gather*}
    The model energy of the Allen-Cahn equation \eqref{eq:AllenCahn_Eq} is 
    \begin{gather*}
		\mathcal{E}(u) := \int_{\Omega} \Big( \frac{\epsilon^{2}}{2} |\nabla u|^{2} + F(u) \Big) dx,
    \end{gather*}
    which satisfies the following energy dissipation law:
    \begin{gather}
		\frac{d}{dt} \mathcal{E}(u(t)) = - \int_{\Omega} u_{t}^{2} dx \leq 0. 
		\label{eq:EN-dissipation-Law}
    \end{gather}

    Phase field models, including the Allen-Cahn equation, have been widely applied to complex moving interface problems in materials science and fluid dynamics.
    Over the past forty years, numerical simulations of phase field models have been extensively studied, leading to a wide array of spatial discretization techniques including finite element methods, alongside other approaches such as finite difference, finite volume, and Fourier-spectral methods.
    Various splitting techniques and stabilizers, coupled with classical time integrators, have been proposed to handle the potential functions in phase field models, resulting in numerous efficient and stable algorithms, as summarized in the recent survey paper \cite{DF20_Elsevier}.

    Convex splitting technique is widely employed to ensure the discrete energy dissipation law \cite{SWWW12_SIAM_JNA,SLL17_JCP,Eyr98_MRSSP,WWL09,GLWW14_JCP}. 
    The main idea is to split the potential function into the convex and concave parts, treated implicitly and explicitly, respectively. 
    Linearly-implicit stabilizers are another popular way for gradient flows, leading to linear solvers at each time step and ensuring unconditional energy stability \cite{Yan09_DCDS_SerB,LQ17_JSC,SS17_JSC,YCWW18_CCP}.
    Exponential integrators, transfering phase field equations into equivalent integral formulations, also achieve energy stability \cite{DZ05_BIT_NM,BM11_NM,JZD15_CMS}. 
    Additionally, Lagrange multiplier approach is an alternative way to construct stable numerical schemes \cite{GT13_JCP}, and two new classes of time-stepping schemes, invariant energy quadratization (IEQ) \cite{Yan19_CMAME,YJ17_CMAME,PCZY23_CMAME} and scalar auxiliary variable (SAV) schemes \cite{CS18_SIAM_JSC,ALL19_SIAM_JSC}, have been recently proposed. 
    The variable step backward differentiation formulation (BDF2) method has been thoroughly investigated to show energy stability for the Allen-Cahn equation, under the ratio of successive time steps satisfying $\tau < (3+\sqrt{17})/2$ \cite{LTZ20_SIAM_JNA}.

    Fully discrete interior penalty discontinuous Galerkin methods for phase field models were developed in \cite{FL15_IMA_JNA,FLX16_SIAM_JNA} to derive error estimates dependent on the polynomial order of the reciprocal model parameter ($1/\epsilon$) rather than the exponential order. 
    Reduced-order finite element methods, based on the proper orthogonal decomposition method, can improve efficiency by reducing the dimension of the solution space while preserving the discrete energy dissipation law \cite{LWSZ21_CMA}. 
    Explicit hybrid finite difference and operator splitting methods have demonstrated the ability to obtain pointwise boundedness of the numerical solutions \cite{JK18_JCAM}.

    In this paper, we consider a family of two-step time-stepping schemes proposed by Dahlquist, Liniger and Nevanlinna (hence the DLN method) which is unconditional $G$-stable (non-linear stable) and second order accurate under non-uniform time grids \cite{DLN83_SIAM_JNA,Dah76_Tech_RIT,Dah78_BIT,Dah78_AP_NYL}. 
    To the best of our knowledge, \textit{the DLN method is the only multi-step scheme possessing these two properties}.
    Recently the DLN method has been applied to various fluid models and its fine properties have been confirmed by some benchmark test problems \cite{LPQT21_NMPDE,QHPL21_JCAM,Pei24_NM,SP23_arXiv,QCWLL23_ANM}.  
    Moreover, the DLN implementation can be simplified by adding a few lines of code to the commonly-used backward Euler scheme \cite{LPT21_AML}.
    To utilize the unconditional $G$-stability to a large extent, some adaptive DLN algorithms have been proposed to balance the time accuracy and computational costs \cite{LPT23_ACSE}.

    Given the initial value problem: 
    \begin{gather}
		\label{eq:IVP}
		y'(t) = g(t,y(t)), \ \ \ 0 \leq t \leq T, \ \ \ y(0) = y_{0},
    \end{gather}
    with $y\!\!: [0,T] \rightarrow \mathbb{R}^{d}$, $g\!\!: [0,T] \times \mathbb{R}^{d} \rightarrow \mathbb{R}^{d}$ and $y_{0} \in \mathbb{R}^{d}$, the two-step, variable time-stepping DLN method, with parameter $\theta \in [0,1]$, takes the form 
    \begin{gather}
		\label{eq:1legDLN}
		\tag{DLN}
		\sum_{\ell =0}^{2}{\alpha _{\ell }}y_{n-1+\ell }
		= \widehat{k}_{n} g \Big( \sum_{\ell =0}^{2}{\beta _{\ell }^{(n)}}t_{n-1+\ell } ,
		\sum_{\ell =0}^{2}{\beta _{\ell }^{(n)}}y_{n-1+\ell} \Big)
		,
		\qquad
		n=1,\ldots,N-1.
    \end{gather}
    Here we denote $\{ 0 = t_{0} < t_{1} < \cdots < t_{N-1} <t_{N}=T \}_{n=0}^{N}$ as the time grid on the interval $[0,T]$, with $k_{n} = t_{n+1} - t_{n}$ as the time step, $\widehat{k}_{n} = {\alpha_{2}}k_{n} - {\alpha_{0}}k_{n-1}$ as the average time step, and $y_{n}$ as the discrete DLN solution approximating $y(t_{n})$.
    The coefficients in \eqref{eq:1legDLN} take the form
    \begin{gather}\label{eq:alphabeta}
		\begin{bmatrix}
			\alpha _{2} \vspace{0.2cm} \\
			\alpha _{1} \vspace{0.2cm} \\
			\alpha _{0} 
		\end{bmatrix}
		= 
		\begin{bmatrix}
			\frac{1}{2}(\theta +1) \vspace{0.2cm} \\
			-\theta \vspace{0.2cm} \\
			\frac{1}{2}(\theta -1)
		\end{bmatrix}, \ \ \ 
		\begin{bmatrix}
			\beta _{2}^{(n)}  \vspace{0.2cm} \\
			\beta _{1}^{(n)}  \vspace{0.2cm} \\
			\beta _{0}^{(n)}
		\end{bmatrix}
		= 
		\begin{bmatrix}
			\frac{1}{4}\Big(1+\frac{1-{\theta }^{2}}{(1+{%
					\varepsilon _{n}}{\theta })^{2}}+\varepsilon_{n}^{2}\frac{\theta (1-{%
					\theta }^{2})}{(1+{\varepsilon _{n}}{\theta })^{2}}+\theta \Big)\vspace{0.2cm%
			} \\
			\frac{1}{2}\Big(1-\frac{1-{\theta }^{2}}{(1+{\varepsilon _{n}}{%
					\theta })^{2}}\Big)\vspace{0.2cm} \\
			\frac{1}{4}\Big(1+\frac{1-{\theta }^{2}}{(1+{%
					\varepsilon _{n}}{\theta })^{2}}-\varepsilon_{n}^{2}\frac{\theta (1-{%
					\theta }^{2})}{(1+{\varepsilon _{n}}{\theta })^{2}}-\theta \Big)%
		\end{bmatrix}.
    \end{gather}
    The step variability $\varepsilon _{n} = (k_n - k_{n-1})/(k_n + k_{n-1})$ represents the variation between two consecutive step sizes. 
    On uniform time grids, step variability $\varepsilon_{n}$ reduces to $0$, and the coefficients 
    $[\beta _{2}^{(n)},\, \!\beta _{1}^{(n)}, \, \!\beta _{0}^{(n)}]$ simplify to $\frac14[2+\theta-\theta^2,\,  2\theta^2,\,  2-\theta-\theta^2]$.	

    Various efficient ways to address the non-linear term of the Allen-Cahn model \eqref{eq:AllenCahn_Eq} in pursuit of unconditional energy stability have been designed.
    In the paper, we consider the implicit modified Crank-Nicolson method first, followed by the SAV method, which have been widely studied in the last decade.  
    The implicit modified Crank-Nicolson scheme has been shown to provide unconditional energy stability in the Allen-Cahn model \cite{DN91_SIAMJNA,CMS11_MC,SY10_DCDS,XLWB19_CMAME}.  
    The SAV scheme for gradient flows \cite{SXY18_JCP,SX18_SIAMJNA,SXY19_SIAMRev} achieves unconditional stability of the model energy by introducing an auxiliary scalar function, which eliminates the need to solve non-linear algebraic equations.
    In this context, we investigate the partially implicit modified DLN algorithm and the DLN-SAV algorithm, combined with the finite element spatial discretization, for robust simulations of the Allen Cahn model \eqref{eq:AllenCahn_Eq}. 
    The main contributions of this work are to:
    \begin{itemize}
		\item[$\bullet$] Construct the variable step modified DLN and DLN-SAV finite element algorithms based on the DLN refactorization process (simplifying the DLN implementation by adding filters on the backward Euler scheme);
		\item[$\bullet$] Provide unconditional stability of model energy under arbitrary time grids for both modified DLN and DLN-SAV schemes;
		\item[$\bullet$] Present a rigorous proof that the fully discrete modified DLN algorithm is \textit{1st order} accurate in time under \textit{arbitrary time steps} and \textit{2nd order} accurate in time under \textit{uniform time grids}; 
		\item[$\bullet$] Design time adaptive algorithms for both modified DLN and DLN-SAV schemes by the local truncation error (LTE) criterion: utilizing certain explicit time-stepping methods to estimate the LTE and adjusting time step size based on the estimator.
    \end{itemize}

    The rest of the paper is organized as follows. We provide necessary preliminaries and notations in Section \ref{sec:Prelimimaries}. 
    In Section \ref{sec:DLN-CSS}, we study the partially implicit modifed DLN scheme. We prove that the discrete model energy is unconditional stable under variable time steps. Error estimate of the resulting fully discrete method is also provided. 
    In Section \ref{sec:DLN-SAV}, we propose the variable time-stepping DLN-SAV scheme. We prove that the corresponding model energy satisfies the discrete energy dissipation law.
    Moreover, we discuss how the DLN-SAV algorithm can be implemented by a simplified refactorization process: adding a few lines of code on the backward Euler-SAV (BE-SAV) scheme. 
    In Section \ref{sec:Adaptivity}, the corresponding time adaptive algorithms (base on LTE criterion) are presented to improve computational efficiency. 	
    We estimate the LTE by the numerical solutions of the explicit AB2-like method (a revised two-step Adams-Bashforth method) with little or almost no extra computational costs. 
    Several one- and two-dimensional numerical tests are presented in Section \ref{sec:NumercalTests} to confirm the main conclusions of the paper. Conclusion remarks are provided in Section \ref{sec7}. Proofs of some lemmas in Section \ref{sec:Prelimimaries} are given in the appendices.

    \section{Preliminaries and Notations}
    \label{sec:Prelimimaries}
    We start by introducing some necessary notations. 
    Let $H^{r}(\Omega)$ be the usual standard Sobolev space $W^{r,2}(\Omega)$ with the norm $\| \cdot \|_{r}$. When $r=0$, this reduces to the inner product space $L^{2}(\Omega)$ with the $L^{2}$ norm $\| \cdot \|$ and the standard $L^{2}$ inner product $(\cdot, \cdot)$. 
    For spatial discretization, we denote $X_{s}^{h} \subset H^{1}(\Omega)$ as a standard finite element space on $\Omega$ with highest polynomial degree $s$ and a mesh diameter $h>0$. 
    Throughout the paper, $C$ denotes a generic positive constant independent of $h$ and $k$.

    For an arbitrary sequence $\{ z_{n} \}_{n=0}^{\infty}$, we define
    \begin{gather}
		z_{n,\theta} = \frac{1+\theta}{2} z_{n} +\frac{1-\theta}{2}z_{n-1}, \,\,
		z_{n,\alpha} = \sum_{\ell = 0}^{2} \alpha_{\ell} z_{n-1+\ell}, \,\, 
		z_{n,\beta} = \sum_{\ell = 0}^{2} \beta_{\ell}^{(n)} z_{n-1+\ell}, \label{eq:nota-seq1} \\
		z_{n,\ast} = \beta_{2}^{(n)} \Big[ \big(1 + \frac{k_{n}}{k_{n-1}} \big) z_{n}
		- \frac{k_{n}}{k_{n-1}} z_{n-1} \Big]
		+ \beta_{1}^{(n)} z_{n} + \beta_{0}^{(n)} z_{n-1}, \quad 
		n \in \mathbb{N},
    \label{eq:nota-seq}
    \end{gather}
    with $\alpha_\ell$ and $\beta^{(n)}_\ell$ defined in \eqref{eq:alphabeta}.

    The following lemmas on stability and consistency of the DLN method \eqref{eq:1legDLN} would play a key role in the numerical analysis.
    \begin{lemma}
		\label{lemma:Gstab}
		For any sequence $\{ y_{n} \}_{n=0}^{N}$ in $L^{2}(\Omega)$ over $\mathbb{R}$, 
		$n \in \{ 1,2, \cdots, N-1 \}$ and $\theta \in [0,1]$, we have
        \begin{align}
				\big(y_{n,\alpha}, \, y_{n,\beta} \big)
				\!=\!
				\begin{Vmatrix}
				{y_{n\!+\!1}} \\
				{y_{n}}
				\end{Vmatrix}
				_{G(\!\theta \!)}^{2} \!-\!
				\begin{Vmatrix}
				{y_{n}} \\
				{y_{n\!-\!1}}
				\end{Vmatrix}
				_{G(\!\theta \!)}^{2} 
				\!+\! \Big\|\!\sum_{\ell \!=\!0}^{2}{\gamma_{\ell }^{(n)}}y_{n\!-\!1\!+\!\ell} \Big\|^{2},
                \label{eq:G-stab}
		\end{align}
        where the $G$-norm $\| \cdot \|_{G(\theta)}$ is defined by
		\begin{align}  
				\begin{Vmatrix}
				u \\
				v\end{Vmatrix}_{G(\theta)}^{2}
				:=& \frac{1}{4} (1+{\theta})\| u \|^{2}
				+ \frac{1}{4} (1 - \theta ) \| v \|^{2},
				\qquad
				\forall u,v\in L^{2}(\Omega),
        \label{eq:G-norm}
		\end{align}
		and the coefficients $\gamma_{\ell }^{(n)}$ are given by
		\begin{align}
				\gamma_{1}^{(n)}=-\frac{\sqrt{\theta \left( 1-{\theta }^{2}\right) }}{\sqrt{2}%
					(1+\varepsilon _{n}\theta )},\quad \gamma_{2}^{(n)}=-\frac{1-\varepsilon _{n}}{2}%
				\gamma_{1}^{(n)},\quad \gamma_{0}^{(n)}=-\frac{1+\varepsilon _{n}}{2}\gamma_{1}^{(n)}.
		\end{align}
    \end{lemma}
	The derivation of \eqref{eq:G-stab} follows from algebraic calculation, and is omitted here to save space. 
  	Following this, the variable time-stepping DLN method in \eqref{eq:1legDLN} is unconditional $G$-stable. (We refer to \cite{Dah76_Tech_RIT,Dah78_BIT} for the definition of $G$-stability for multi-step, time-stepping methods.) 

    \begin{lemma}
    \label{lemma:consistency-1st}
    Let $u(\cdot,t)$ be the mapping from $[0,T]$ to $L^{2}(\Omega)$, and $u_{n}$ be the function $u(\cdot,t_{n})$ in $L^{2}(\Omega)$. 
    Assuming the mapping $u(\cdot,t)$ is smooth with respect to $t$, then for any $\theta \in [0,1)$, we have
    \begin{align}
      &\big\| (u_{n+1,\theta} - u_{n,\theta})^{2} \big\|^{2} \leq C(\theta) (k_{n} + k_{n-1})^{3} 
      \int_{t_{n-1}}^{t_{n+1}} \| u_{t} \|_{L^{4}}^{4} dt, \label{eq:consist-1st-eq1} \\
      &\Big\| \frac{u_{n+1,\theta} + u_{n,\theta}}{2} - u(t_{n,\beta}) \Big\|^{2} \leq C(\theta) (k_{n} + k_{n-1})
      \int_{t_{n-1}}^{t_{n+1}} \| u_{t} \|^{2} dt. \label{eq:consist-1st-eq2} 
    \end{align}
    When $\theta = 1$, the DLN method reduces to the midpoint rule, and we have
    \begin{align}
      \big\| (u_{n\!+\!1,\theta} \!-\! u_{n,\theta})^{2} \big\|^{2} 
      =& \big\| (u_{n\!+\!1} \!-\! u_{n})^{2} \big\|^{2} 
      \!\leq\! C k_{n}^{3} \! \int_{t_{n}}^{t_{n\!+\!1}} \! \| u_{t} \|_{L^{4}}^{2} dt, \label{eq:consist-1st-eq3} \\
      \Big\| \frac{u_{n\!+\!1,\theta} \!+\! u_{n,\theta}}{2} \!-\! u(t_{n,\beta}) \Big\|^{2} 
      =& \Big\| \frac{u_{n\!+\!1} \!+\! u_{n}}{2} \!-\! u \big( \frac{t_{n\!+\!1} \!+\! t_{n}}{2} \big) \Big\|^{2} 
      \!\leq\! C k_{n}^{3} \! \int_{t_{n}}^{t_{n\!+\!1}} \! \| u_{t} \|^{2} dt. \label{eq:consist-1st-eq4} 
    \end{align}
    Under uniform time grids with constant time step $k$, \eqref{eq:consist-1st-eq2} becomes 
    \begin{align}
      \Big\| \frac{u_{n+1,\theta} + u_{n,\theta}}{2} - u(t_{n,\beta}) \Big\|^{2} 
      \leq C(\theta) k^{3} \int_{t_{n-1}}^{t_{n+1}} \| u_{tt} \|^{2} dt.
    \label{eq:consist-1st-eq2-const}
    \end{align}
    \end{lemma}

    \begin{lemma}
        \label{lemma:consistency-2nd}
      Let $u(\cdot,t)$ be the mapping from $[0,T]$ to $H^{r}(\Omega)$, and $u_{n}$ be the function $u(\cdot,t_{n})$ in $H^{r}(\Omega)$. Assuming the mapping $u(\cdot,t)$ is smooth with respect to $t$, then for any $\theta \in [0,1)$, we have
      \begin{align}
      \| u_{n,\beta} - u(t_{n,\beta}) \|_{r}^{2} 
      \leq& C(\theta) (k_{n} + k_{n-1})^{3} \int_{t_{n-1}}^{t_{n+1}} \| u_{tt} \|_{r}^{2} dt, 
      \label{eq:consist-2nd-eq1} \\
      \Big\| \frac{u_{n,\alpha}}{\widehat{k}_{n}} - u_{t}(t_{n,\beta})\Big\|_{r}^{2}  
      \leq& C(\theta) (k_{n} + k_{n-1})^{3} \displaystyle\int_{t_{n-1}}^{t_{n+1}} \| u_{ttt} \|_{r}^{2} dt. \label{eq:consist-2nd-eq2}
      \end{align}
      When $\theta = 1$, the DLN method reduces to the midpoint rule, and we have
      \begin{align}
      \big\| u_{n,\beta} - u(t_{n,\beta}) \big\|_{r}^{2} 
      =& \Big\| \frac{u_{n+1} + u_{n}}{2} - u \big( \frac{t_{n+1} + t_{n}}{2} \big)\Big\|_{r}
      \leq C k_{n}^{3} \int_{t_{n}}^{t_{n+1}} \| u_{tt} \|_{r}^{2} dt, 
      \label{eq:consist-2nd-eq3} \\
      \Big\| \frac{u_{n,\alpha}}{\widehat{k}_{n}} - u_{t}(t_{n,\beta})\Big\|_{r}^{2}  
      =& \Big\| \frac{u_{n+1} - u_{n}}{k_{n}} - u_{t} \big( \frac{t_{n+1} + t_{n}}{2} \big) \Big\|_{r}
      \leq C k_{n}^{3} \displaystyle\int_{t_{n}}^{t_{n+1}} \| u_{ttt} \|_{r}^{2} dt. \label{eq:consist-2nd-eq4}
      \end{align}
    \end{lemma}
    The proof of these two Lemmas can be found at Appendix \ref{appendixA} and \ref{appendixB}.

    \section{The Modified DLN Algorithm}
    \label{sec:DLN-CSS}
    Let $u_{n}^{h} \in X_{s}^{h}$ be the numerical approximation of $u(x,t)$ in \eqref{eq:AllenCahn_Eq} at time $t_{n}$. 
    The modified DLN finite element algorithm for the Allen-Cahn equation \eqref{eq:AllenCahn_Eq} is formulated as the following: 
    given $u_{n-1}^{h},u_{n}^{h} \in X_{s}^{h}$, find $u_{n+1}^{h} \in X_{s}^{h}$ such that 
    \begin{gather}
        \Big(\frac{u_{n,\alpha}^{h}}{\widehat{k}_{n}}, v^{h} \Big)
        + \epsilon^{2} \big(\nabla u_{n,\beta}^{h}, \nabla v^{h} \big)
        + \big( \widetilde{f}(u_{n+1,\theta}^{h},u_{n,\theta}^{h}), v^{h} \big) = 0, \label{eq:DLN-CSS-Alg}
    \end{gather}
    holds for all $v^{h} \in X_{s}^{h}$, where $u_{n,\alpha}^{h}$, $u_{n,\beta}^{h}$ and $u_{n,\theta}^{h}$ are defined in \eqref{eq:nota-seq1} and 
    \begin{align*}
		&\widetilde{f}(u_{n+1,\theta}^{h},u_{n,\theta}^{h}) \\
        =& \frac{1}{4}  
        \Big\{ \big[ \big( u_{n+1,\theta}^{h} \big)^{3} + \big( u_{n+1,\theta}^{h} \big)^{2}u_{n,\theta}^{h} 
		+ u_{n+1,\theta}^{h} \big( u_{n,\theta}^{h} \big)^{2} + \big( u_{n,\theta}^{h} \big)^{3} \big] 
        - 2 \big( u_{n+1,\theta}^{h} + u_{n,\theta}^{h} \big) \Big\} \\
		=& 
		\begin{cases}
			\displaystyle \frac{F(u_{n+1,\theta}^{h}) - F(u_{n,\theta}^{h})}{u_{n+1,\theta}^{h} - u_{n,\theta}^{h}} 
			& \text{if } u_{n+1,\theta}^{h} \neq u_{n,\theta}^{h}, \\
			\displaystyle f \Big(\frac{u_{n+1,\theta}^{h} + u_{n,\theta}^{h}}{2} \Big) 
			& \text{if } u_{n+1,\theta}^{h} = u_{n,\theta}^{h}.
		\end{cases}
    \end{align*}

    \begin{remark}
        \label{remark1}
        The function $f(u)$ is approximated by $\widetilde{f}$ numerically, which has the following properties.
        \begin{itemize}
            \item[$\bullet$] $\widetilde{f} (u, v)$ is continuous and $\partial_{u} \widetilde{f} (u,v), \partial_{v} \widetilde{f} (u,v)$ exist. 
            Note that the solution $u(x,t)$ of the Allen-Cahn model \eqref{eq:AllenCahn_Eq} satisfies the maximum principle: if the initial value and boundary conditions are bounded by some constant $\Gamma$, then the solution at a later time are also bounded by $\Gamma$, i.e.,
            \begin{gather*}
                \| u(\cdot, t) \|_{\infty} < \Gamma, \qquad \forall t>0.
            \end{gather*} 
            Following the maximum principle, there exists $L > 0$ such that
            $\| f'(u) \|_{\infty}\leq L$,
            therefore, $\widetilde{f} (u, v)$ is Lipschitz-continuous for both components. 
            \item[$\bullet$] For $\theta \in [0,1)$,
            $\widetilde{f} (u_{n+1,\theta}, u_{n,\theta})$ is a first order approximation to $f(u(t_{n,\beta}))$ in time under variable time steps and a second order approximation in time under constant time steps. This can be observed from \eqref{eq:consist-1st-eq2}, \eqref{eq:consist-1st-eq2-const} and 
            \begin{align}
                &\widetilde{f} (u_{n+1,\theta}, u_{n,\theta}) - f(u(t_{n,\beta})) 
                \label{eq:2nd-approx-f} \\
                =& f'(\xi_{n}) \Big( \frac{u_{n+1,\theta} + u_{n,\theta}}{2} - u(t_{n,\beta}) \Big) 
                + \frac{1}{24} f'' \big( \frac{u_{n+1,\theta} + u_{n,\theta}}{2} \big) (u_{n+1,\theta} - u_{n,\theta})^{2}, \notag
            \end{align}
            which follows from Taylor's expansion, where $\xi_{n}$ is between $u_{n}$ and $u_{n,\beta}$. 
            For $\theta = 1$, $\widetilde{f} (u_{n+1,\theta}, u_{n,\theta})$ is always a second order approximation to $f(u(t_{n,\beta}))$ in time, which is implied by \eqref{eq:consist-1st-eq4} and \eqref{eq:2nd-approx-f}.
        \end{itemize}
    \end{remark}

    \begin{remark}
        The modified DLN algorithm \eqref{eq:DLN-CSS-Alg} with $\theta = 1$ reduces to the modified Crank-Nicolson algorithm for the Allen-Cahn equation studied in \cite{XLWB19_CMAME}.
    \end{remark}
        
    \begin{remark}
        \label{remark2}
        Based on the modified DLN method, we can formulate the DLN Convex Splitting scheme (DLN-CSS). As introduced in \cite{XLWB19_CMAME}, we split the potential function $F$ into the difference between two convex functions: the implicit part $F_1(u) = \frac{u^4 +1}{4}$ and the explicit part $F_2(u) = \frac{u^2}{2}$. 
        Then we set
        \begin{equation}\label{eq:21}
              \widehat{f}(u_{n+1,\theta}^{h}, u_{n,\theta}^{h}) = \frac{F_1(u_{n+1,\theta}^{h,\tt{IM}}) - F_1(u_{n,\theta}^{h})}{u_{n+1,\theta}^{h,\tt{IM}} - u_{n,\theta}^{h}} - \frac{F_2(u_{n+1,\theta}^{h,\tt{EX}}) - F_2(u_{n,\theta}^{h})}{u_{n+1,\theta}^{h,\tt{EX}} - u_{n,\theta}^{h}},
        \end{equation}
        where $u_{n+1,\theta}^{h,\tt{IM}} = u_{n+1,\theta}^{h}$ and $u_{n+1,\theta}^{h,\tt{EX}} = \frac{1+\theta}{2} \Big[ \big(1 + \frac{k_{n}}{k_{n-1}} \big) u_{n}^{h}
        - \frac{k_{n}}{k_{n-1}} u_{n-1}^{h} \Big] +\frac{1-\theta}{2}u_{n}^{h}$. 
        The DLN-CSS scheme is obtained by replacing $\widetilde{f}(u_{n+1,\theta}^{h}, u_{n,\theta}^{h})$ in \eqref{eq:DLN-CSS-Alg} with $\widehat{f}$ in \eqref{eq:21}. 
        
        We test the DLN-CSS scheme on the 1D travelling wave example in Section \ref{sec:NumTest:1D}. The numerical results are similar to those of the modified DLN scheme, and are omitted to save space. 
    \end{remark}

    \subsection{Unconditional Stability in Model Energy}
    We define the corresponding discrete model energy for the modified DLN algorithm \eqref{eq:DLN-CSS-Alg} at time $t_{n}$ as
    \begin{align}
        \mathcal{E}_{n}^{\tt{MOD}} := \epsilon^2 
		\begin{Vmatrix}
			\nabla u_{n}^{h} \\
			\nabla u_{n-1}^{h}
		\end{Vmatrix}_{G(\theta)}^{2} + \int_{\Omega} F(u_{n,\theta}^{h}) dx.
		\label{eq:EN-DLN-CSS-Alg}
    \end{align} 
    The following theorem presents the discrete energy dissipation law satisfied by the numerical solution.
    \begin{theorem}
        The model energy of the variable time-stepping modified DLN algorithm \eqref{eq:DLN-CSS-Alg} satisfies the discrete energy dissipation law:
        \begin{gather}
            \mathcal{E}_{n+1}^{\tt{MOD}} \leq \mathcal{E}_{n}^{\tt{MOD}}, \qquad n = 1,2, \cdots, N-1,
            \label{eq:EN-law-DLN-CSS}
        \end{gather}
    thus the modified DLN algorithm \eqref{eq:DLN-CSS-Alg} is unconditional stable in model energy.
    \end{theorem}
    \begin{proof}
        Setting $v^h = u_{n,\alpha}^{h} \big(= u_{n+1,\theta}^{h} - u_{n,\theta}^{h} \big)$ in \eqref{eq:DLN-CSS-Alg} leads to 
        \begin{gather}
          \frac{1}{\widehat{k}_{n}} \| u_{n,\alpha}^{h} \|^{2}
          + \epsilon^2 
          \begin{Vmatrix}
            \nabla u_{n+1}^{h} \\
            \nabla u_{n}^{h}
          \end{Vmatrix}_{G(\theta)}^{2}
          - \epsilon^2 
          \begin{Vmatrix}
            \nabla u_{n}^{h} \\
            \nabla u_{n-1}^{h}
          \end{Vmatrix}_{G(\theta)}^{2}
          + \epsilon^{2}  \Big\| \nabla \Big(\sum_{\ell \!=\!0}^{2}{\gamma_{\ell }^{(n)}}u_{n-1+\ell}^{h} \Big) \Big\|^{2} 
          \notag \\
          + \big( \widetilde{f}(u_{n+1,\theta}^{h},u_{n,\theta}^{h}), u_{n+1,\theta}^{h} - u_{n,\theta}^{h}  \big) = 0, \label{eq:Stab-CSS-eq1} 
        \end{gather}
        following the $G$-stability identity \eqref{eq:G-stab}.
        Then we use the fact
        \begin{gather*}
          \widetilde{f}(u_{n+1,\theta}^{h},u_{n,\theta}^{h}) \big( u_{n+1,\theta}^{h} - u_{n,\theta}^{h} \big) 
          = F(u_{n+1,\theta}^{h}) - F(u_{n,\theta}^{h}),
        \end{gather*}
        to obtain
        \begin{gather*}
          \mathcal{E}_{n+1}^{\tt{MOD}} - \mathcal{E}_{n}^{\tt{MOD}} + \frac{1}{\widehat{k}_{n}} \| u_{n,\alpha}^{h} \|^{2}
          + \epsilon^{2}  \Big\| \nabla \Big(\sum_{\ell \!=\!0}^{2}{\gamma_{\ell }^{(n)}}u_{n-1+\ell}^{h} \Big) \Big\|^{2} = 0, 
        \end{gather*}
        which implies \eqref{eq:EN-law-DLN-CSS}.
    \end{proof}

    \subsection{Error Analysis} 
    This subsection is devoted to the error estimate of the fully discrete modified DLN algorithm \eqref{eq:DLN-CSS-Alg}. 
    We remind that $\{ \beta_{\ell}^{(n)} \}_{\ell=0}^{2}$ are uniformly bounded for any fixed $\theta \in [0,1]$, i.e. 
    $|\beta_{\ell}^{(n)}| < C_{\beta}(\theta)$ for all $\ell,n$, and define the following discrete Bochner function space for error analysis 
	\begin{align}
		&\ell_{\theta}^{2} (0,N;H^{r}(\Omega)) 
    \label{eq:def-discrete-norm} \\
    &:=
		\Big\{ \! u(\cdot,t) \in H^{r}(\Omega) \!:\! \| u \|_{\ell_{\theta}^{2} (0,N,H^{r}(\Omega))} 
		\!=\! \Big( \sum_{n\!=\!0}^{N\!-\!1} \! k_{n} \| u(t_{n,\theta}) \|_{r}^{2} \Big)^{\frac{1}{2}} \!<\! \infty \! \Big\}. \notag 
	\end{align}

    \begin{theorem} \label{theorem2}
        Let $u_{n}^{h} \in X_{s}^{h}$ denote numerical solutions of the modified DLN algorithm in \eqref{eq:DLN-CSS-Alg} and $u_{n}$ denote exact solutions of the Allen-Cahn equation in \eqref{eq:AllenCahn_Eq} at time $t_{n}$. We assume 
        \begin{gather*}
            u(\cdot, t) \in H^{s+1}(\Omega), \qquad \forall t \in [0,T], \\
            u \in \ell_{\theta}^{2}(0,N;H^{s+1}(\Omega)), \quad
            u_{t} \in L^{2}(0,T;H^{s+1}(\Omega)) \cap L^{4} (0,T;L^{4}(\Omega)), \\
            u_{tt} \in L^{2}(0,T;H^{1}(\Omega)), \quad
            u_{ttt} \in L^{2}(0,T;L^{2}(\Omega)),
        \end{gather*} 
        there exists $L>0$ such that $\|f'(u) \|_{\infty} < L$ for all $t \in [0,T]$,
        and time step sizes and time step ratios satisfy 
        \begin{gather}
            k_{\rm{max}} = \max_{0 \leq n \leq N-1} \{k_{n} \} 
            < \frac{1}{12\sqrt{3} C_{\beta}(\theta) L}, 
            \label{eq:time-max-condi} \\
            C_{\ell} < \frac{k_{n}}{k_{n-1}} < C_{u}, \qquad n = 1, 2, \cdots, N-1,
            \label{eq:time-ratio-condi}
        \end{gather}
        for some $C_{\ell},C_{u}>0$. 
        Then we have: for $\theta \in [0,1)$
        \begin{gather}
            \max_{0 \leq n \leq N} \{ \| u_{n}^{h} - u_{n} \| \} 
            \leq \mathcal{O} (k_{\rm{max}}, h^{s+1}), \label{eq:conclusion-vari} \\
            \Big( \frac{\epsilon^{2}}{2} \sum_{n=1}^{N-1} \widehat{k}_{n} 
            \| \nabla ( u_{n,\beta}^{h} - u_{n,\beta} )\|^{2} \Big)^{1/2}
            \leq  \mathcal{O} (k_{\rm{max}}, h^{s}), \notag 
        \end{gather} 
        and for $\theta = 1$
        \begin{gather}
            \max_{0 \leq n \leq N} \{ \| u_{n}^{h} - u_{n} \| \} 
            \leq \mathcal{O} (k_{\rm{max}}^{2}, h^{s+1}), \label{eq:conclusion-vari-theta1} \\
            \Big( \frac{\epsilon^{2}}{2} \sum_{n=1}^{N-1} \widehat{k}_{n} 
            \| \nabla (u_{n,\beta} - u_{n,\beta}^{h} )\|^{2} \Big)^{1/2} \leq  \mathcal{O} (k_{\rm{max}}^{2}, h^{s}). \notag 
        \end{gather} 
        Under uniform time grids with constant time step $k$, \eqref{eq:conclusion-vari} becomes
        \begin{gather}
            \max_{0 \leq n \leq N} \{ \| u_{n}^{h} - u_{n} \| \} \leq \mathcal{O} (k^{2}, h^{s+1}),
            \label{eq:conclusion-const} \\
            \Big( \frac{\epsilon^{2}}{2} \sum_{n=1}^{N-1} k
            \| \nabla (u_{n,\beta} - u_{n,\beta}^{h} )\|^{2} \Big)^{1/2} \leq \mathcal{O} (k^{2}, h^{s}). \notag 
        \end{gather}
    \end{theorem}
    \begin{proof}
        We start by defining the following Ritz projection operator $\Pi$: given $v \in H^{1}(\Omega)$, $\Pi (v) \in X_{s}^{h}$ satisfies
        \begin{gather} \label{projection}
		    \big( \nabla \Pi (v) , \nabla w^{h} \big) = (\nabla v, \nabla w^{h}), \qquad
		    \big( \Pi (v) - v, 1 \big) = 0, \qquad \forall w^{h} \in X_{s}^{h}.
        \end{gather}
        The following approximation theorem \cite{Cia02_SIAM,Tho06_Springer} holds for the Ritz projection: for any $v \in H^{r}(\Omega)$ and its Ritz projection $\Pi (v) \in X_{s}^{h}$, we have
        \begin{gather}
		    \| v - \Pi (v) \| + h \| v - \Pi (v) \|_{1} \leq Ch^{\min\{r,s+1 \}} \| v \|_{\min\{r,s+1 \}}.
		    \label{eq:ApproxProjector}
        \end{gather}

        We define the error of $u$ at time $t_{n}$ to be 
		\begin{gather*}
			e_{n} := u_{n} - u_{n}^{h} = \eta_{n} + \phi_{n}^{h}, \qquad \eta_{n} := u_{n} - \Pi (u_{n}), \qquad \phi_{n}^{h} : = \Pi (u_{n}) - u_{n}^{h}.
		\end{gather*}
        The exact solution $u$ at $t_{n,\beta}$ satisfies 
		\begin{gather}
			\big( u_{t} (t_{n,\beta}), v^{h} \big) + \epsilon^{2} \big( \nabla u(t_{n,\beta}), \nabla v^{h} \big)
			+ \big( f(u(t_{n,\beta})), v^{h} \big) = 0, \quad v^{h} \in X^{h}. \label{eq:exact}
		\end{gather} 
        Subtracting \eqref{eq:DLN-CSS-Alg} from \eqref{eq:exact} yields
        \begin{align}
			\Big(\frac{u_{n,\alpha}^h}{\widehat{k}_{n}} - u_{t} (t_{n,\beta}), v^{h} \Big) 			
			&+ \epsilon^{2} \Big( \nabla \big(u_{n,\beta}^h - u(t_{n,\beta}) \big), \nabla v^{h} \Big)  \notag \\
			&+ \big(\widetilde{f}(u_{n+1,\theta}^{h},u_{n,\theta}^{h}) - f(u(t_{n,\beta})), v^{h} \big) \notag \\
            = \Big(\frac{u_{n,\alpha}-e_{n,\alpha}}{\widehat{k}_{n}} - u_{t} (t_{n,\beta}), v^{h} \Big) 			
			&+ \epsilon^{2} \Big( \nabla \big(u_{n,\beta} - e_{n,\beta} - u(t_{n,\beta}) \big), \nabla v^{h} \Big)  
            \label{eq:DLN-CSS-error-eq1_new} \\
			&+ \big(\widetilde{f}(u_{n+1,\theta}^{h},u_{n,\theta}^{h}) - f(u(t_{n,\beta})), v^{h} \big)= 0. 
            \notag 
		\end{align}
        The definition of $\Pi$ leads to $\big( \nabla \eta_{n,\beta}, \nabla v^{h} \big)=0$, and \eqref{eq:DLN-CSS-error-eq1_new} can be rewritten as
		\begin{align}
			& \Big( \frac{\phi_{n,\alpha}^{h}}{\widehat{k}_{n}}, v^{h} \Big)
			+ \epsilon^{2} \big( \nabla \phi_{n,\beta}^{h}, \nabla v^{h} \big) \label{eq:DLN-CSS-error-eq1} \\
			=& \Big(\frac{u_{n,\alpha}}{\widehat{k}_{n}} - u_{t} (t_{n,\beta}), v^{h} \Big) 
			- \Big( \frac{\eta_{n,\alpha}}{\widehat{k}_{n}}, v^{h} \Big) 
			+ \epsilon^{2} \Big( \nabla \big(u_{n,\beta} - u(t_{n,\beta}) \big), \nabla v^{h} \Big) \notag \\
			&+ \Big(\widetilde{f}(u_{n+1,\theta}^{h},u_{n,\theta}^{h}) - f(u(t_{n,\beta})), v^{h} \Big). \notag
		\end{align}
        We set $v^{h} = \phi_{n,\beta}^{h}$ in \eqref{eq:DLN-CSS-error-eq1} and use the $G$-stability identity in \eqref{eq:G-stab} to obtain  
		\begin{align}
			& \frac{1}{\widehat{k}_{n}}\Big(
			\begin{Vmatrix}
				{\phi_{n\!+\!1}^{h}} \\
				{\phi_{n}^{h}}
			\end{Vmatrix}_{G(\!\theta \!)}^{2} \!-\!
			\begin{Vmatrix}
				{\phi_{n}^{h}} \\
				{\phi_{n\!-\!1}^{h}}
			\end{Vmatrix}_{G(\!\theta \!)}^{2} 
			\!+\! \Big\|\!\sum_{\ell \!=\!0}^{2}{\gamma_{\ell }^{(n)}} \phi_{n\!-\!1\!+\!\ell}^{h} \!\Big\|^{2} \Big)
			+ \epsilon^{2} \| \nabla \phi_{n,\beta}^{h} \|^{2} 
			\label{eq:DLN-CSS-error-eq2} \\
			=& \Big(\frac{u_{n,\alpha}}{\widehat{k}_{n}} - u_{t} (t_{n,\beta}), \phi_{n,\beta}^{h} \Big)
			- \Big( \frac{\eta_{n,\alpha}}{\widehat{k}_{n}}, \phi_{n,\beta}^{h} \Big) 
			+ \epsilon^{2} \Big( \nabla \big(u_{n,\beta} - u(t_{n,\beta}) \big), \nabla \phi_{n,\beta}^{h} \Big) \notag \\
			&+ \Big(\widetilde{f}(u_{n+1,\theta}^{h},u_{n,\theta}^{h}) - f(u(t_{n,\beta})), \phi_{n,\beta}^{h} \Big). \notag 
		\end{align}

        Next, we provide a bound for each term on the right-hand side of \eqref{eq:DLN-CSS-error-eq2}. Utilizing \eqref{eq:consist-2nd-eq2}, Young's inequality and the fact $|\beta_{\ell}^{(n)}| < C_{\beta}(\theta)$ ($\ell = 0,1,2$) leads to
		\begin{align}
			&\Big(\frac{u_{n,\alpha}}{\widehat{k}_{n}} - u_{t} (t_{n,\beta}), \phi_{n,\beta}^{h} \Big) \label{eq:DLN-CSS-Error-term1} \\
			\leq& C(\theta,\delta,L) \Big\| \frac{u_{n,\alpha}}{\widehat{k}_{n}} - u_{t} (t_{n,\beta}) \Big\|^{2} + \frac{\delta L^{2} (1+\theta)}{3} \big( \| \phi_{n+1}^{h} \|^{2} + \| \phi_{n}^{h} \|^{2} + \| \phi_{n-1}^{h} \|^{2}\big)
			\notag \\
			\leq& C(\theta,\delta,L) k_{\rm{max}}^{3} \int_{t_{n-1}}^{t_{n+1}} \| u_{ttt} \|^{2} dt
			+ \frac{\delta L^{2} (1+\theta)}{3} \big( \| \phi_{n+1}^{h} \|^{2} + \| \phi_{n}^{h} \|^{2} + \| \phi_{n-1}^{h} \|^{2}\big),	\notag 
		\end{align}
		where $\delta > 0$ will be decided later and $C(\theta,\delta,L)$ is a constant only depending on $\theta,\delta,L$.
        By Cauchy-Schwarz inequality, Young's inequality and the projection error \eqref{eq:ApproxProjector}
		\begin{align}
			\frac{1}{\widehat{k}_{n}} \big( \eta_{n,\alpha}, \phi_{n,\beta}^{h}  \big) 
			\!\leq& \! \frac{C(\theta,\delta,L)}{\widehat{k}_{n}^{2}} \! \big\| \eta_{n,\alpha} \big\|^{2} 
			\!+\! \frac{\delta L^{2} (1+\theta) }{3} \big( \| \phi_{n+1}^{h} \|^{2} + \| \phi_{n}^{h} \|^{2} + \| \phi_{n-1}^{h} \|^{2}\big)
			\notag  \\
			\leq& \!\frac{C\!(\theta,\!\delta,\!L) h^{2s\!+\!2}}{\widehat{k}_{n}^{2}} 
			\!\big\| u_{n,\alpha} \big\|_{s\!+\!1}^{2} 
			\!+\! \frac{\delta L^{2} (1\!+\!\theta) }{3} \! \big( \| \phi_{n\!+\!1}^{h} \|^{2} \!+\! \| \phi_{n}^{h} \|^{2} \!+\! \| \phi_{n\!-\!1}^{h} \|^{2}\big)\notag \\
			\leq& \! \frac{C\!(\theta,\delta,L) h^{2s\!+\!2}}{\widehat{k}_{n}^{2}} \!\! \big( \! \| u_{n\!+\!1} \!-\! u_{n} \|_{s\!+\!1}^{2} \!+\! \| u_{n} \!-\! u_{n\!-\!1} \|_{s\!+\!1}^{2} \! \big) 
      		\notag \\
			&+\! \frac{\delta L^{2} (1\!+\!\theta) }{3} \big( \| \phi_{n\!+\!1}^{h} \|^{2} \!+\! \| \phi_{n}^{h} \|^{2} \!+\! \| \phi_{n\!-\!1}^{h} \|^{2}\big). 
            \label{eq:DLN-CSS-Error-term2-eq1}
		\end{align}
        By H$\rm{\ddot{o}}$lder's inequality, we have
		\begin{align}
			\| u_{n+1} - u_{n} \|_{s\!+\!1}^{2} 
			&\leq \Big\| \int_{t_{n}}^{t_{n+1}} u_{t}(\cdot,t) dt \Big\|_{s+1}^{2}
			\leq k_{n} \int_{t_{n}}^{t_{n+1}} \| u_{t} \|_{s+1}^{2} dt, \label{eq:DLN-CSS-Error-term2-eq2} 
		\end{align}
		therefore, \eqref{eq:DLN-CSS-Error-term2-eq1} becomes 
		\begin{align}
			\frac{1}{\widehat{k}_{n}} \!\big( \eta_{n,\alpha}, \phi_{n,\beta}^{h} \big)
			\leq& \frac{C(\delta,\theta,L) h^{2s\!+\!2}}{\widehat{k}_{n}} \! \int_{t_{n-1}}^{t_{n+1}} \| u_{t} \|_{s+1}^{2} dt \label{eq:DLN-CSS-Error-term2} \\
			& + \frac{\delta L^{2} (1+\theta)}{3} \big( \| \phi_{n+1}^{h} \|^{2} + \| \phi_{n}^{h} \|^{2} + \| \phi_{n-1}^{h} \|^{2}\big).  \notag 
		\end{align}
        By Cauchy-Schwarz inequality and \eqref{eq:consist-2nd-eq1}
		\begin{align}
			\epsilon^{2} \Big( \nabla \big(u_{n,\beta} - u(t_{n,\beta}) \big), \nabla \phi_{n,\beta}^{h} \Big)
			\leq& \epsilon^{2} \Big\| \nabla \big(u_{n,\beta} - u(t_{n,\beta}) \big) \Big\| 
			\| \nabla \phi_{n,\beta}^{h} \| \label{eq:DLN-CSS-Error-term3} \\
			\leq& C(\theta) \epsilon^{2} k_{\rm{max}}^{3} \int_{t_{n-1}}^{t_{n+1}} \| \nabla u_{tt} \|^{2} dt
			+ \frac{\epsilon^{2}}{2} \| \nabla \phi_{n,\beta}^{h} \|^{2}. \notag 
		\end{align}

        For the non-linear term, utilizing Cauchy Schwarz inequality, triangle inequality and Young's inequality leads to
	    \begin{align}
		    &\Big(\widetilde{f}(u_{n+1,\theta}^{h},u_{n,\theta}^{h}) - f(u(t_{n,\beta})), \phi_{n,\beta}^{h} \Big) 
		    \label{eq:DLN-CSS-error-nonlinear}
		    \\
		    \leq& C_{\beta}(\theta) \big\| \widetilde{f}(u_{n+1,\theta}^{h},u_{n,\theta}^{h}) - f(u(t_{n,\beta})) \big\| 
		    \big( \| \phi_{n+1}^{h} \| + \| \phi_{n}^{h} \| + \| \phi_{n-1}^{h} \| \big) \notag \\
		    \leq& \frac{9 C_{\beta}^{2}(\theta)}{4 \delta L^{2} (1+\theta)} 	
		    \big\| \widetilde{f}(u_{n+1,\theta}^{h},u_{n,\theta}^{h}) - f(u(t_{n,\beta})) \big\|^{2} \notag \\
		    &\qquad + \frac{\delta L^{2}(1+\theta)}{9} \big( \| \phi_{n+1}^{h} \| + \| \phi_{n}^{h} \| + \| \phi_{n-1}^{h} \| \big)^{2} \notag \\
		    \leq& \frac{9 C_{\beta}^{2}(\theta)}{4 \delta L^{2} (1+\theta)} 	
		    \big\| \widetilde{f}(u_{n+1,\theta}^{h},u_{n,\theta}^{h}) - f(u(t_{n,\beta})) \big\|^{2} \notag \\
		    &\qquad + \frac{\delta L^{2}(1+\theta)}{3} \big( \| \phi_{n+1}^{h} \|^{2} + \| \phi_{n}^{h} \|^{2} + \| \phi_{n-1}^{h} \|^{2} \big). \notag
	    \end{align}
        Considering the first term on the right-hand side, we have 
	    \begin{align}
		    &\big\| \widetilde{f}(u_{n+1,\theta}^{h},u_{n,\theta}^{h}) - f(u(t_{n,\beta})) \big\|^{2} 
            \leq  3\big\| \widetilde{f}(u_{n+1,\theta}^{h},u_{n,\theta}^{h}) 
		    - \widetilde{f}(u_{n+1,\theta},u_{n,\theta}^{h}) \big\|^{2}  \\
		    &+ 3\big\| \widetilde{f}(u_{n+1,\theta},u_{n,\theta}^{h}) 
		    - \widetilde{f}(u_{n+1,\theta},u_{n,\theta}) \big\|^{2}  
		    + 3\big\| \widetilde{f}(u_{n+1,\theta},u_{n,\theta}) - f(u(t_{n,\beta})) \big\|^{2} . \notag
	    \end{align}
        By the assumption $\| f'(u) \|_{\infty} < L$, $\widetilde{f}(\cdot,\cdot)$ are Lipschitz-continuous for both components \cite{SY10_DCDS}. Thus  
        \begin{align}
			&\big\| \widetilde{f}(u_{n\!+\!1,\theta}^{h},u_{n,\theta}^{h}) 
			\!-\! \widetilde{f}(u_{n\!+\!1,\theta},u_{n,\theta}^{h}) \big\|^{2}  			
			\leq L^{2} \| u_{n+1,\theta}^{h} \!-\! u_{n+1,\theta} \|^{2}   \label{eq:DLN-CSS-error-nonlinear-1} \\
			\leq\,& L^{2} \big\| \eta_{n+1,\theta} + \phi_{n+1,\theta}^{h} \big\|^{2} 
			\leq C L^{2} h^{2s\!+\!2} \!\big\| u_{n\!+\!1,\theta} \big\|_{s+1}^{2} 
			\!+\! 2 L^{2} \big\| \phi_{n+1,\theta}^{h} \big\|^{2}  \notag \\
			\leq\,& C L^{2} h^{2s\!+\!2} \! \Big( \big\| u_{n+1,\theta} - u(t_{n+1,\theta}) \big\|_{s+1}^{2} \!+\! \| u(t_{n+1,\theta}) \|_{s+1}^{2} \Big) 
			\!+\! 2 L^{2} \big\| \phi_{n+1,\theta}^{h} \big\|^{2}  
			\notag \\
			\leq& C(\theta,L) h^{2s\!+\!2} \! \Big( k_{n} \!\! \int_{t_{n}}^{t_{n+1}}\!\! \| u_{t} \|_{s+1}^{2} dt + \| u(t_{n+1,\theta})\|_{s+1}^{2} \Big) \notag \\
            &+ (1+\theta)^{2} L^{2} \big( \| \phi_{n+1}^{h} \|^{2} + \| \phi_{n}^{h} \|^{2} \big), \notag
	    \end{align}
        where the last inequality follows from H$\rm{\ddot{o}}$lder's inequality as in \eqref{eq:DLN-CSS-Error-term2-eq2}. Similarly, we have
		\begin{align}
			& \big\| \widetilde{f}(u_{n+1,\theta},u_{n,\theta}^{h}) 
			- \widetilde{f}(u_{n+1,\theta},u_{n,\theta}) \big\|^{2}  
			\label{eq:DLN-CSS-error-nonlinear-2} \\
			\leq& C(\theta,L) h^{2s\!+\!2} \! \Big( k_{n\!-\!1}\! \int_{t_{n\!-\!1}}^{t_{n}} \| u_{t} \|_{s\!+\!1}^{2} dt \!+\! \| u(t_{n,\theta})\|_{s\!+\!1}^{2} \Big) \!+\! (1\!+\!\theta)^{2} L^{2} \big( \| \phi_{n}^{h} \|^{2} \!+\! \| \phi_{n\!-\!1}^{h} \|^{2} \big). \notag
		\end{align}
        We use \eqref{eq:consist-1st-eq1}-\eqref{eq:consist-1st-eq2} in Lemma \ref{lemma:consistency-1st} and \eqref{eq:2nd-approx-f} to obtain
		\begin{align}
			&\| \widetilde{f}(u_{n+1,\theta},u_{n,\theta}) - f ( u(t_{n,\beta}) ) \|^{2}  \label{eq:DLN-CSS-error-nonlinear-3} \\
			\leq& C(L) \big\| \frac{u_{n+1,\theta} + u_{n,\theta}}{2} - u(t_{n,\beta}) \big\|^{2} 
			+ C \| (u_{n+1,\theta} - u_{n,\theta})^{2} \|^{2}  \notag \\
			\leq& C(\theta,L) (k_{n} + k_{n-1}) \int_{t_{n-1}}^{t_{n+1}} \| u_{t} \|^{2} dt
			+ C(\theta) (k_{n} + k_{n-1})^{3} \int_{t_{n-1}}^{t_{n+1}} \| u_{t} \|_{L^{4}}^{4} dt. \notag  
		\end{align}
        Combining \eqref{eq:DLN-CSS-error-nonlinear} - \eqref{eq:DLN-CSS-error-nonlinear-3}, we have 
		\begin{align}
			& \Big(\widetilde{f}(u_{n+1,\theta}^{h},u_{n,\theta}^{h}) - f(u(t_{n,\beta})), \phi_{n,\beta}^{h} \Big) 
			\label{eq:DLN-CSS-error-nonlinear-final} \\
			\leq& C(\theta,\delta,L) h^{2s+2} \! \Big( k_{\rm{max}} \int_{t_{n-1}}^{t_{n+1}} \| u_{t} \|_{s+1}^{2} dt 
			+ \| u(t_{n+1,\theta})\|_{s+1}^{2} + \| u(t_{n,\theta})\|_{s+1}^{2} \Big) \notag \\
			& + C(\theta,\delta,L) k_{\rm{max}} \int_{t_{n-1}}^{t_{n+1}} \| u_{t} \|^{2} dt
			+ C(\theta,\delta,L) k_{\rm{max}}^{3} \int_{t_{n-1}}^{t_{n+1}} \| u_{t} \|_{L^{4}}^{4} dt \notag \\
			&+ \frac{27 (1+\theta) C_{\beta}^{2}(\theta)}{4 \delta} \big( \| \phi_{n+1}^{h} \|^{2} + 2 \| \phi_{n}^{h} \|^{2} 
			+ \| \phi_{n-1}^{h} \|^{2}\big) \notag \\
      		&+ \frac{\delta L^{2} (1+\theta)}{3} \big( \| \phi_{n+1}^{h} \|^{2} + \| \phi_{n}^{h} \|^{2} + \| \phi_{n-1}^{h} \|^{2}\big). \notag 
		\end{align}

        Now, we are ready to derive the error estimate. We combine \eqref{eq:DLN-CSS-error-eq2}, \eqref{eq:DLN-CSS-Error-term1}, \eqref{eq:DLN-CSS-Error-term2}, 
		\eqref{eq:DLN-CSS-Error-term3} and \eqref{eq:DLN-CSS-error-nonlinear-final}, multiply both sides by $\widehat{k}_{n}$ and use the time ratio condition in \eqref{eq:time-ratio-condi} to obtain 
        \begin{align} 
			& \begin{Vmatrix}
				{\phi_{n\!+\!1}^{h}} \\
				{\phi_{n}^{h}}
			\end{Vmatrix}_{G(\!\theta \!)}^{2} \!-\!
			\begin{Vmatrix}
				{\phi_{n}^{h}} \\
				{\phi_{n\!-\!1}^{h}}
			\end{Vmatrix}_{G(\!\theta \!)}^{2} 
			\!+\! \Big\|\sum_{\ell \!=\!0}^{2}{\gamma_{\ell }^{(n)}} \phi_{n-1+\ell}^{h} \Big\|^{2} 
			+ \frac{\epsilon^{2}}{2} \widehat{k}_{n} \| \nabla \phi_{n,\beta}^{h} \|^{2} 
			\label{eq:DLN-CSS-error-eq3} \\
			\leq& (1+\theta) \Big( \frac{27 C_{\beta}^{2}(\theta)}{4 \delta} 
			+ \delta L^{2}  \Big) \widehat{k}_{n} 
			\big( \| \phi_{n+1}^{h} \|^{2} + 2 \| \phi_{n}^{h} \|^{2} + \| \phi_{n-1}^{h} \|^{2}\big) \notag \\
			&+ C(\theta,\delta,L) k_{\rm{max}}^{4} \int_{t_{n-1}}^{t_{n+1}} \| u_{ttt} \|^{2} dt
			+ C(\delta,\theta,L) h^{2s+2} \! \int_{t_{n-1}}^{t_{n+1}} \| u_{t} \|_{s+1}^{2} dt \notag \\
			&+ C(\theta) \epsilon^{2} k_{\rm{max}}^{4} \int_{t_{n-1}}^{t_{n+1}} \| \nabla u_{tt} \|^{2} dt  
			+ C(\theta,\delta,L) h^{2s+2} k_{\rm{max}}^{2} \int_{t_{n-1}}^{t_{n+1}} \| u_{t} \|_{s+1}^{2} dt \notag \\
			&+ C(\theta,\delta,L) h^{2s+2} \big( \widehat{k}_{n}\| u(t_{n+1,\theta})\|^{2}_{s+1} + \widehat{k}_{n}\| u(t_{n,\theta})\|^{2}_{s+1}  \big)  \notag \\ 
			&+ C(\theta,\delta,L) k_{\rm{max}}^{2} \int_{t_{n-1}}^{t_{n+1}} \| u_{t} \|^{2} dt
			+ C(\theta,\delta,L) k_{\rm{max}}^{4} \int_{t_{n-1}}^{t_{n+1}} \| u_{t} \|_{L^{4}}^{4} dt, \notag
		\end{align}
        We sum \eqref{eq:DLN-CSS-error-eq3} over $n$ from $1$ to $N-1$, use the definition of $G(\theta)$-norm in \eqref{eq:G-norm}, drop the non-negative term $\| \phi_{N-1}^{h} \|^{2}$, $\| \sum_{\ell=0}^{2} \gamma_{\ell}^{(n)} \phi_{n-1+\ell}^{h} \|^{2}$ and use the time ratio condition in \eqref{eq:time-ratio-condi} to obtain 
		\begin{align}
			&\Big[ \frac{1+\theta}{4} - (1+\theta) \Big( \frac{27 C_{\beta}^{2}(\theta)}{4 \delta} + \delta L^{2}  \Big) k_{\rm{max}} \Big] \| \phi_{N}^{h} \|^{2}
			+ \frac{\epsilon^{2}}{2} \sum_{n=1}^{N-1} \widehat{k}_{n} \| \nabla \phi_{n,\beta}^{h} \|^{2} 
			\label{eq:DLN-CSS-error-eq4} \\
			\leq& C(\theta, \delta, L,C_{\ell},C_{u}) \sum_{n=0}^{N-1}  k_{n} \| \phi_{n}^{h} \|^{2} 
			+ C(\theta,\delta,L) k_{\rm{max}}^{4}  \| u_{ttt} \|_{L^{2}(0,T;L^{2})}^{2}
			\notag \\
			&+ C(\delta,\theta,L) h^{2s+2} (1+k_{\rm{max}}^{2}) \| u_{t} \|_{L^{2}(0,T;H^{s+1})}^{2}  
			+ C(\theta) \epsilon^{2} k_{\rm{max}}^{4} \| \nabla u_{tt} \|_{L^{2}(0,T;L^{2})}^{2}
			\notag \\
			&+ C(\theta,\delta,L,C_{\ell},C_{u}) h^{2s+2} \| u \|_{\ell_{\theta}^{2}(0,N;H^{s+1})}^{2} 
			+ C(\theta,\delta,L) k_{\rm{max}}^{2} \| u_{t} \|_{L^{2}(0,T;L^{2})}^{2} \notag \\
			&+ C(\theta,\delta,L) k_{\rm{max}}^{4} \| u_{t} \|_{L^{4}(0,T;L^{4})}^{4} + \frac{1+\theta}{4}\| \phi_{1}^{h} \|^{2} + \frac{1-\theta}{4}\| \phi_{0}^{h} \|^{2}. \notag  
		\end{align}

        To estimate $\| \phi_{N}^{h} \|^{2}$, we need 
		\begin{gather*}
			\frac{1+\theta}{4} - (1+\theta) \Big( \frac{27 C_{\beta}^{2}(\theta)}{4 \delta} + \delta L^{2}  \Big) k_{\rm{max}} > 0.
		\end{gather*}
		We set $\delta = \frac{3\sqrt{3} C_{\beta}(\theta)}{2L}$ to relax the upper bound for $k_{\rm{max}}$ and obtain the time step restriction in \eqref{eq:time-max-condi}.
		By the discrete Gr$\rm{\ddot{o}}$nwall inequality \cite{HR90_SIAM_NA} and the fact that the maximum time step $k_{\rm{max}} < 1$, we have 
		\begin{align}
			&\| \phi_{N}^{h} \|^{2}
			+ \frac{\epsilon^{2} C(\theta,L)}{2} \sum_{n=1}^{N-1} \widehat{k}_{n} \| \nabla \phi_{n,\beta}^{h} \|^{2} 
			\label{eq:DLN-CSS-phi} \\
			\leq& \exp \Big(C(\theta,L,C_{\ell},C_{u}) T \Big) \Big\{ C(\theta,L) k_{\rm{max}}^{4} \| u_{ttt} \|_{L^{2}(0,T;L^{2})}^{2}
			+ C(\theta,L) h^{2s+2} \| u_{t} \|_{L^{2}(0,T;H^{s+1})}^{2} \notag \\
			&\qquad \qquad \qquad \quad + C(\theta) \epsilon^{2} k_{\rm{max}}^{4} \| \nabla u_{tt} \|_{L^{2}(0,T;L^{2})}^{2}
			+ C(\theta,L) k_{\rm{max}}^{2} \| u_{t} \|_{L^{2}(0,T;L^{2})}^{2} \notag \\
			&\qquad \qquad \qquad \quad + \!C(\theta,L,C_{\ell},C_{u}) h^{2s+2} \| u \|_{\ell_{\theta}^{2}(0,N;H^{s+1})}^{2}
			\!+\! C(\theta,L) k_{\rm{max}}^{4} \| u_{t} \|_{L^{4}(0,T;L^{4})}^{4} \notag \\
			&\qquad \qquad \qquad \quad + C(\theta) \big(\| \phi_{1}^{h} \|^{2} + \| \phi_{0}^{h} \|^{2} \big) \Big\}, \notag 
		\end{align} 
        which leads to
		\begin{align}
			&\| \phi_{N}^{h} \| + \Big( \frac{\epsilon^{2}}{2} \sum_{n=1}^{N-1} \widehat{k}_{n} \| \nabla \phi_{n,\beta}^{h} \|^{2} \Big)^{1/2}  
			\label{eq:DLN-CSS-phi-final} \\
			\leq& C(\theta,L,T,C_{\ell},C_{u}) \Big\{ \! k_{\rm{max}}^{2} \| u_{ttt} \|_{L^{2}(0,T;L^{2})} 
			\!+\! h^{s+1} \| u_{t} \|_{L^{2}(0,T;H^{s+1})} 
			\!+\! \epsilon k_{\rm{max}}^{2} \| \nabla u_{tt} \|_{L^{2}(0,T;L^{2})}
			\notag \\
			&\qquad \qquad \qquad 
			+ k_{\rm{max}} \| u_{t} \|_{L^{2}(0,T;L^{2})} 
			+ h^{s+1} \| u \|_{\ell_{\theta}^{2}(0,N;H^{s+1})}
			+ k_{\rm{max}}^{2} \| u_{t} \|_{L^{4}(0,T;L^{4})}^{2}
			\notag \\
			&\qquad \qquad \qquad + \| \phi_{1}^{h} \| + \| \phi_{0}^{h} \| \Big\} \notag.
		\end{align}
        By \eqref{eq:DLN-CSS-phi-final}, triangle inequality and \eqref{eq:ApproxProjector}, we have \eqref{eq:conclusion-vari}.

        Lastly, we comment on the case of uniform time grids with constant time step $k$. Under such case, utilizing \eqref{eq:consist-1st-eq1} and \eqref{eq:consist-1st-eq2-const} in Lemma \ref{lemma:consistency-1st}, Eq. \eqref{eq:DLN-CSS-error-nonlinear-3} becomes
		\begin{align*}
		&\| \widetilde{f}(u_{n+1,\theta},u_{n,\theta}) - f ( u(t_{n,\beta}) ) \|^{2}  \\
		\leq& C(L) \Big\| \frac{u_{n+1,\theta} + u_{n,\theta}}{2} - u(t_{n,\beta}) \Big\|^{2} + C \| (u_{n+1,\theta} - u_{n,\theta})^{2} \|^{2}  \notag \\
		\leq& C(\theta,L) k^{3} \int_{t_{n-1}}^{t_{n+1}} \| u_{tt} \|^{2} dt
		+ C(\theta) k^{3} \int_{t_{n-1}}^{t_{n+1}} \| u_{t} \|_{L^{4}}^{4} dt. \notag  
		\end{align*}
		Following similar arguments, we have \eqref{eq:conclusion-const}.

    \end{proof}

    \begin{remark}
        We have proven that the modified DLN algorithm is first-order accurate in time under arbitrary time step sequences. However, all numerical tests in Section \ref{sec:NumercalTests} demonstrate that numerical solutions on both uniform and non-uniform time grids converge at second-order in time. 
    \end{remark}

    \section{The DLN-SAV algorithm} 
	\label{sec:DLN-SAV}
	In this section, we consider the energy stable SAV approach, and present the DLN-SAV algorithm. For the Allen-Cahn model \eqref{eq:AllenCahn_Eq}, we introduce the scalar auxiliary variable of the form
	\begin{align*}
		r(t) &= \sqrt{E(u(x,t)) + C_{0}}, 
	\end{align*}
    with 
    \[    E(u(x,t)) = \int_{\Omega} F(u(x,t)) dx = \frac14 \int_{\Omega} (u(x,t)^2-1)^2 dx\geq 0,
    \]
    where $C_{0}$ is a positive constant (can be any positive constant) to ensure $r(t)>0$. In our numerical experiments, we often take $C_0=0$. 
    Then the Allen-Cahn model in \eqref{eq:AllenCahn_Eq} can be equivalently written as
	\begin{gather*}
		\begin{cases}
		\displaystyle u_{t} - \epsilon^{2} \Delta u + \frac{r}{\sqrt{E(u) + C_{0}}} f(u) = 0,  \qquad  x \in \Omega, \ 0 < t \leq T, \\
		\displaystyle r_t = \frac{1}{2\sqrt{E(u) + C_{0}}} \int_{\Omega} f(u)u_t dx, \qquad  \ 0 < t \leq T.
		\end{cases} 
	\end{gather*}

    \subsection{The algorithm}
	Let $u_{n}^{h}$, $r_{n}^{h}$ be fully-discrete numerical solutions of $u(x,t_{n})$, $r(t_{n})$ respectively.
	The fully-discrete weak formulation for the DLN-SAV algorithm with the finite element spatial discretization is: given $u_{n-1}^{h},u_{n}^{h} \in X_{s}^{h}$ and $r_{n-1}^{h},r_{n}^{h}$, we find $u_{n+1}^{h} \in X_{s}^{h}$ and $r_{n+1}^{h}$ such that for all $v^{h} \in X_{s}^{h}$
	\begin{gather}
	\begin{cases}
		\displaystyle \Big(\frac{ u_{n,\alpha}^{h}}{\widehat{k}_{n}}, v^{h} \Big)
		+ \epsilon^{2} \big(\nabla u_{n,\beta}^{h}, \nabla v^{h} \big)
		+ \frac{r_{n,\beta}^{h}}{ \sqrt{E \big( u_{n,\ast}^{h} \big) + C_{0} }} \big( f (u_{n,\ast}^{h}), v^{h} \big) = 0, \\
		\displaystyle \frac{ r_{n,\alpha}^{h}}{\widehat{k}_{n}} 
		= \frac{1}{2 \sqrt{E \big( u_{n,\ast}^{h} \big) + C_{0} }} \Big( f (u_{n,\ast}^{h}), \frac{u_{n,\alpha}^{h}}{\widehat{k}_{n}} \Big),
	\end{cases}
	\label{eq:DLN-SAV-Alg-weak} 
	\end{gather} 
    where $u_{n,\ast}^{h}$ defined in \eqref{eq:nota-seq} is the explicit second-order approximation to $u(t_{n,\beta})$ in time.

    The variable time-stepping DLN scheme can be simplified by the refactorization process: adding a few lines of codes to the \textit{backward Euler (BE) scheme} to obtain the DLN solutions (see \cite{LPT21_AML} for details).
    The DLN-SAV algorithm can be equivalently rewritten as the following process: \\
	\textit{Step 1.} Pre-process
	\begin{gather*}
        k_{n}^{\tt{BE}} = b^{(n)} \widehat{k}_{n}, \ \ 
		u_{n}^{h,\tt{old}} = a_{1}^{(n)} u_{n}^{h} + a_{0}^{(n)} u_{n-1}^{h} \in X_{s}^{h}, \ \ 
		r_{n}^{h,\tt{old}} = a_{1}^{(n)} r_{n}^{h} + a_{0}^{(n)} r_{n-1}^{h},
	\end{gather*} 
    where the coefficients for the pre-process are
    \begin{gather*}
		a_{1}^{(n)} = \beta_{1}^{(n)} - \alpha_{1} \beta_{2}^{(n)} / \alpha_{2}, \,\,  
		a_{1}^{(n)} = \beta_{0}^{(n)} - \alpha_{0} \beta_{2}^{(n)} / \alpha_{2}, \,\,
		b^{(n)} = \beta_{2}^{(n)} / \alpha_{2}.
	\end{gather*}
    \textit{Step 2.} BE solver to solve $u_{n+1}^{h,\tt{temp}} \in X_{s}^{h}$ and $r_{n+1}^{h,\tt{temp}}$ such that for all $v^{h} \in X_{s}^{h}$ 
	\begin{gather}
		\begin{cases}
			\displaystyle \Big(\frac{ u_{n+1}^{h,\tt{temp}} - u_{n}^{h,\tt{old}}}{k_{n}^{\tt{BE}}}, v^{h} \Big)
			\!+\! \epsilon^{2} \big(\nabla u_{n+1}^{h,\tt{temp}}, \nabla v^{h} \big)
			\!+\! \frac{r_{n+1}^{h,\tt{temp}}}{ \sqrt{E \big( u_{n,\ast}^{h} \big) + C_{0} }} \big( f (u_{n,\ast}^{h}), v^{h} \big) \!=\! 0, \\
			\displaystyle \frac{r_{n+1}^{h,\tt{temp}} - r_{n}^{h,\tt{old}} }{k_{n}^{\tt{BE}}} 
			= \frac{1}{2 \sqrt{E \big( u_{n,\ast}^{h} \big) + C_{0} }} \Big( f (u_{n,\ast}^{h}), \frac{u_{n+1}^{h,\tt{temp}} - u_{n}^{h,\tt{old}}}{k_{n}^{\tt{BE}}} \Big).
		\end{cases}
		\label{eq:BE-SAV-Alg-weak} 
	\end{gather}
    \textit{Step 3.} Post-process
	\begin{gather*}
		u_{n+1}^{h} \!=\! c_{2}^{(n)} u_{n+1}^{h,\tt{temp}} \!+\! c_{1}^{(n)} u_{n}^{h} \!+\! c_{0}^{(n)} u_{n-1}^{h}, 
    \quad
		r_{n+1}^{h} \!=\! c_{2}^{(n)} r_{n+1}^{h,\tt{temp}} \!+\! c_{1}^{(n)} r_{n}^{h} \!+\! c_{0}^{(n)} r_{n-1}^{h},
	\end{gather*}
	where the coefficients for the post-process are 
	\begin{gather*}
		c_{2}^{(n)} = 1/ \beta_{2}^{(n)}, \qquad c_{1}^{(n)}  = - \beta_{1}^{(n)} / \beta_{2}^{(n)}, \qquad 
		c_{0}^{(n)}  = - \beta_{0}^{(n)} / \beta_{2}^{(n)}.
	\end{gather*}
    The above refactorization process can be applied to the modified DLN algorithm in \eqref{eq:DLN-CSS-Alg} as well. 
    However the refactorization process of the modified DLN algorithm does not simplify the implementation much since the partially implicit, non-linear term needs to be refactorized as well, and was omitted in Section \ref{sec:DLN-CSS}. 

    \begin{lemma}
		The DLN-SAV algorithm in \eqref{eq:DLN-SAV-Alg-weak} admits a unique solution. 
	\end{lemma}
    \begin{proof}
        Since the DLN-SAV algorithm in \eqref{eq:DLN-SAV-Alg-weak} is equivalent to the above refactorization process, it suffices to show that the BE solver in \eqref{eq:BE-SAV-Alg-weak} has a unique solution. 
        We denote
		\begin{gather*}
			\varphi_{n}^{h} = \frac{f (u_{n,\ast}^{h})}{ \sqrt{E \big( u_{n,\ast}^{h} \big) + C_{0} }},
		\end{gather*}
		and eliminate $r_{n+1}^{h,\tt{temp}}$ in \eqref{eq:BE-SAV-Alg-weak} to obtain 
		\begin{gather}
			\big( u_{n+1}^{h,\tt{temp}}, v^{h} \big) + \epsilon^{2} k_{n}^{\tt{BE}} \big( \nabla u_{n+1}^{h,\tt{temp}}, \nabla v^{h} \big) + \frac{1}{2} k_{n}^{\tt{BE}} \big( u_{n+1}^{h,\tt{temp}}, \varphi_{n}^{h} \big)
			\big( \varphi_{n}^{h}, v^{h} \big) = (w_{n}^{h}, v^{h}), 
			\notag \\
			w_{n}^{h} = u_{n}^{h,\tt{old}} + \frac{1}{2} k_{n}^{\tt{BE}} \big( u_{n}^{h,\tt{old}}, \varphi_{n}^{h} \big)\varphi_{n}^{h} - k_{n}^{\tt{BE}} r_{n}^{h,\tt{old}} \varphi_{n}^{h}. 
            \label{eq:BE-SAV-Alg-weak-eq2} 
		\end{gather}
        Let $\{ \psi_{\ell} \}_{\ell=1}^{N_d}$ be the basis for $X_{s}^{h}$. We have 
		$\displaystyle u_{n+1}^{h,\tt{temp}} = \sum_{\ell=1}^{N_{d}} U_{\ell}^{(n+1)} \psi_{\ell} $ for some coordinates $\{ U_{\ell}^{(n+1)} \}_{\ell=1}^{N_{d}}$. We define the matrices $M,K \in \mathbb{R}^{N_{d} \times N_{d}}$ as
		\begin{gather*}
			M = \big[ M_{ij} \big]_{N_d \times N_d} = \big[ ( \psi_{j}, \psi_{i} ) \big]_{N_d \times N_d},   \quad 
			K = \big[ K_{ij} \big]_{N_d \times N_d} = \big[ ( \nabla \psi_{j}, \nabla \psi_{i} ) \big]_{N_d \times N_d},
		\end{gather*}
	    and vectors $B^{(n)}, U^{(n+1)}, W^{(n)} \in \mathbb{R}^{N_{d} \times 1}$ as
		\begin{align*}
			B^{(n)} &= \big[ B_{i}^{(n)} \big]_{N_d \times 1} = \big[ ( \varphi_{n}^{h}, \psi_{i} ) \big]_{N_d \times 1},
			\quad 
			U^{(n+1)} = \big[ U_{i}^{(n+1)} \big]_{N_d \times 1}, \\
			W^{(n)} &= \big[ W_{i}^{(n)} \big]_{N_d \times 1} = \big[ (w_{n}^{h}, \psi_{i}) \big]_{N_d \times 1}.
		\end{align*}
        Then the algebraic structure for \eqref{eq:BE-SAV-Alg-weak-eq2} is 
		\begin{gather}
			\big[ M + \epsilon^{2} k_{n}^{\tt{BE}} K + \frac{1}{2} k_{n}^{\tt{BE}} B^{(n)} (B^{(n)})^{T} \big] U^{(n+1)}
			= W^{(n)}. 
			\label{eq:BE-SAV-Algebra}
		\end{gather}
		Since the mass matrix $M>0$, the stiffness matrix $K \geq 0$ and $B^{(n)} (B^{(n)})^{T} \geq 0$, the linear system in \eqref{eq:BE-SAV-Algebra} has a unique solution. 
	    In addition, $r_{n+1}^{h,\tt{temp}}$  is uniquely obtained from the second equation of \eqref{eq:BE-SAV-Alg-weak}. 

    \end{proof}

    \subsection{Numerical implementation}
    The left matrix in \eqref{eq:BE-SAV-Algebra} is far away from a sparse matrix. To efficiently implement the DLN-SAV scheme, we solve \eqref{eq:BE-SAV-Alg-weak-eq2} in the following simpler way. We denote 
	$A \!=\! I \!-\! \epsilon^{2} k_{n}^{\tt{BE}} \!\Delta > 0$ and apply integration by parts to \eqref{eq:BE-SAV-Alg-weak-eq2} to rewrite it as
	\begin{gather}
		\big( A u_{n+1}^{h,\tt{temp}}, v^{h} \big) + \frac{k_{n}^{\tt{BE}}}{2} (u_{n+1}^{h,\tt{temp}}, \varphi_{n}^{h}) (\varphi_{n}^{h}, v^{h}) = (w_{n}^{h}, v^{h}), \quad \forall v^{h} \in X_{s}^{h}.
		\label{eq:BE-SAV-Alg-weak-eq3}
	\end{gather}
    To solve $u_{n+1}^{h,\tt{temp}}$ in \eqref{eq:BE-SAV-Alg-weak-eq3} efficiently, we first solve for $(u_{n+1}^{h,\tt{temp}}, \varphi_{n}^{h})$.
    Given any function $g \in (H^{-1}(\Omega))^{d}$, we define $A^{-1} g \in X_{s}^{h}$ to be the unique finite element solution to the following second-order system
    \begin{gather*}
        \begin{cases}
			u - \epsilon^{2} k_{n}^{\tt{BE}} \Delta u = g,  & x \in \Omega,   \\
			\frac{\partial u}{\partial \vec{n} } = 0,   & \text{on } \partial{\Omega} .
		\end{cases}
    \end{gather*}
    Thus $A^{-1} g \in X_{s}^{h}$ satisfies the following weak form
	\begin{gather}
		\begin{cases}
			\big( A (A^{-1} g), v^{h} \big) = (g,v^{h}), \\
			\big( \frac{\partial}{\partial \vec{n} } (A^{-1} g), v^{h} \big)_{\partial{\Omega}} = 0,
		\end{cases} \qquad \forall v^{h} \in X_{s}^{h}.
		\label{eq:def-2nd-eq}
	\end{gather}
    We use the definition \eqref{eq:def-2nd-eq} and integration by parts to obtain
	\begin{gather}
		\big( A u_{n+1}^{h,\tt{temp}}, A^{-1} \varphi_{n}^{h} \big) = 
		\big( u_{n+1}^{h,\tt{temp}}, A (A^{-1} \varphi_{n}^{h}) \big) 
		= \big( u_{n+1}^{h,\tt{temp}}, \varphi_{n}^{h} \big), 
		\label{eq:BE-SAV-Alg-weak-eq3-terms-computing} \\
		(w_{n}^{h}, A^{-1} \varphi_{n}^{h}) = \big( A (A^{-1}w_{n}^{h}), A^{-1} \varphi_{n}^{h} \big)
		= \big( A^{-1}w_{n}^{h}, A (A^{-1} \varphi_{n}^{h} ) \big)
		= \big( A^{-1}w_{n}^{h}, \varphi_{n}^{h} \big). \notag 
	\end{gather}
    We use \eqref{eq:BE-SAV-Alg-weak-eq3-terms-computing} and set $v^{h} = A^{-1} \varphi_{n}^{h}$ in \eqref{eq:BE-SAV-Alg-weak-eq3} to solve 
	\begin{gather}
		(u_{n+1}^{h,\tt{temp}}, \varphi_{n}^{h}) 
		= \frac{ \big( A^{-1}w_{n}^{h}, \varphi_{n}^{h} \big)}{ 1 + \frac{k_{n}^{\tt{BE}}}{2} \big( A^{-1}\varphi_{n}^{h}, \varphi_{n}^{h} \big)}.
		\label{eq:BE-SAV-Alg-weak-eq4}
	\end{gather}
    To obtain $u_{n+1}^{h,\tt{temp}}$ from \eqref{eq:BE-SAV-Alg-weak-eq3}, it suffices to solve 
	\begin{gather*}
		A u_{n+1}^{h,\tt{temp}} + \frac{k_{n}^{\tt{BE}}}{2} (u_{n+1}^{h,\tt{temp}}, \varphi_{n}^{h}) \varphi_{n}^{h}
		= w_{n}^{h},
	\end{gather*}
	or equivalently, 
	\begin{align}
		u_{n+1}^{h,\tt{temp}} = - \frac{k_{n}^{\tt{BE}}}{2} (u_{n+1}^{h,\tt{temp}}, \varphi_{n}^{h}) A^{-1} \varphi_{n}^{h} + A^{-1} w_{n}^{h}.
		\label{eq:BE-SAV-Alg-weak-eq5}
	\end{align}

    We summarize the simplified implementation in the following steps.
    \begin{itemize}
		\item[(i)] Solving for $A^{-1} \varphi_{n}^{h} \in X_{s}^{h}$ via
		\begin{gather}
			\big( A^{-1} \varphi_{n}^{h}, v^{h} \big) + \epsilon^{2} k_{n}^{\tt{BE}} \big( \nabla (A^{-1} \varphi_{n}^{h}), \nabla v^{h} \big)
			= \big( \varphi_{n}^{h}, v^{h} \big), \quad \forall v^{h} \in X_{s}^{h}. 
			\label{eq:BE-SAV-Alg-weak-eq6}
		\end{gather}
		\item[(ii)] Solving for $A^{-1} w_{n}^{h} \in X_{s}^{h}$ via
		\begin{gather}
			\big( A^{-1} w_{n}^{h}, v^{h} \big) + \epsilon^{2} k_{n}^{\tt{BE}} \big( \nabla (A^{-1} w_{n}^{h} ), \nabla v^{h} \big)
			= \big( w_{n}^{h}, v^{h} \big), \quad \forall v^{h} \in X_{s}^{h}.
			\label{eq:BE-SAV-Alg-weak-eq7}
		\end{gather}
		\item[(iii)] Computing 
		$(u_{n+1}^{h,\tt{temp}}, \varphi_{n}^{h})$ by \eqref{eq:BE-SAV-Alg-weak-eq4}.
            \item[(iv)] Computing $u_{n+1}^{h,\tt{temp}}$ from \eqref{eq:BE-SAV-Alg-weak-eq5} and $r_{n+1}^{h,\tt{temp}}$ from the second equation of \eqref{eq:BE-SAV-Alg-weak}.
	\end{itemize}

    \begin{remark}
        It is straightforward to verify that the solution pair $(u_{n+1}^{h,\tt{temp}}, r_{n+1}^{h,\tt{temp}})$ obtained by the above process is the unique solution to the algorithm in \eqref{eq:BE-SAV-Alg-weak}, and the details are omitted here.
    \end{remark}

    \begin{remark}
        Solving two second-order equations in \eqref{eq:BE-SAV-Alg-weak-eq6} and \eqref{eq:BE-SAV-Alg-weak-eq7} is much more efficient than solving \eqref{eq:BE-SAV-Algebra} directly since the left matrix of the two linear systems in \eqref{eq:BE-SAV-Alg-weak-eq6} and \eqref{eq:BE-SAV-Alg-weak-eq7} becomes a sparse matrix $M + \epsilon^{2} k_{n}^{\tt{BE}} K$.
    \end{remark}

    \subsection{Unconditional Stability in Model Energy}
	We define the discrete model energy for the DLN-SAV algorithm 
    \eqref{eq:DLN-SAV-Alg-weak} at time $t_{n}$ as
	\begin{gather}
		\mathcal{E}_{n}^{\tt{SAV}} = \epsilon^2 \begin{Vmatrix}
			\nabla u_{n}^{h} \\
			\nabla u_{n-1}^{h}
		\end{Vmatrix}_{G(\theta)}^{2} + \frac{1+\theta}{2} (r_{n}^{h})^2 + \frac{1-\theta}{2} (r_{n-1}^{h})^2,
        \label{eq:DiscreteEnergy}
	\end{gather}
	and have the following theorem on its unconditional stability with respect to $\mathcal{E}_{n}^{\tt{SAV}}$.

    \begin{theorem}
		The model energy of the variable time-stepping DLN-SAV algorithm \eqref{eq:DLN-SAV-Alg-weak} satisfies:
		\begin{gather}
		\mathcal{E}_{n+1}^{\tt{SAV}} \leq \mathcal{E}_{n}^{\tt{SAV}}, \qquad n = 1,2, \cdots , N-1,
		\label{eq:EN-law-DLN-SAV}
		\end{gather}
		thus the DLN-SAV algorithm is unconditional stable with respect to this model energy.
	\end{theorem}
    \begin{proof}
        We set $v^{h} = u_{n,\alpha}^{h}$ in the first equation of the DLN-SAV algorithm \eqref{eq:DLN-SAV-Alg-weak}
		\begin{gather}
			\frac{1}{\widehat{k}_{n}} \| u_{n,\alpha}^{h} \|^{2}
			\!+\! \epsilon^{2} \big(\nabla u_{n,\beta}^{h}, \nabla u_{n,\alpha}^{h} \big)
			+ \frac{r_{n,\beta}^{h}}{\sqrt{E \big( u_{n,\ast}^{h} \big) + C_0 }} \big( f (u_{n,\ast}^{h}), u_{n,\alpha}^{h} \big)
			= 0. 
			\label{eq:EN-DLN-SAV-Alg-Eq1}
		\end{gather}
		We multiply the second equation of the DLN-SAV algorithm \eqref{eq:DLN-SAV-Alg-weak} by $2 \widehat{k}_{n} r_{n,\beta}$ 
		and obtain
		\begin{gather}
			2 r_{n,\alpha}^{h} r_{n,\beta}^{h} 
			= \frac{r_{n,\beta}^{h}}{\sqrt{E \big( u_{n,\ast}^{h} \big) + C_0 }} \big( f (u_{n,\ast}^{h}), u_{n,\alpha}^{h} \big).
			\label{eq:EN-DLN-SAV-Alg-Eq2}
		\end{gather}
		We combine \eqref{eq:EN-DLN-SAV-Alg-Eq1} and \eqref{eq:EN-DLN-SAV-Alg-Eq2} and use the $G$-stability identity \eqref{eq:G-stab} to obtain
		\begin{gather*}
			\frac{1}{\widehat{k}_{n}} \! \| u_{n,\alpha}^{h} \|^{2} + \mathcal{E}_{n+1}^{\tt{SAV}} - \mathcal{E}_{n}^{\tt{SAV}}
			\!+\! {\epsilon }^2 \Big\| \nabla \! \Big(\sum_{\ell \!=\!0}^{2}{\gamma_{\ell }^{(n)}}u_{n\!-\!1\!+\!\ell}^{h} \Big) \! \Big\|^{2}
			\!+\! 2 \Big(\sum_{\ell \!=\!0}^{2}{\gamma_{\ell }^{(n)}} r_{n\!-\!1\!+\!\ell}^{h} \Big)^{2} \! = \! 0, 
		\end{gather*}
		which results in the unconditional stability property \eqref{eq:EN-law-DLN-SAV}. 

    \end{proof}
    \begin{remark}
		The inequality in \eqref{eq:EN-law-DLN-SAV} is the discrete version of the energy dissipation law \eqref{eq:EN-dissipation-Law} 
        and the discrete energy $\mathcal{E}_{n}^{\tt{SAV}}$ in \eqref{eq:DiscreteEnergy} is an approximation of the model energy $\mathcal{E}$. 
	\end{remark}

    \section{Time Adaptivity}
	\label{sec:Adaptivity}
	In this section, we discuss adaptive DLN algorithms for solving the Allen-Cahn equation, to better take advantage of this variable time stepping method. Here we consider adaptive DLN algorithms by using the LTE criterion, which involve two essential parts:
	\begin{itemize}
		\item[$\bullet$] the estimator for LTE,
		\item[$\bullet$] the time step controller.
	\end{itemize}

    An explicit, variable step AB2-like scheme\footnote[1]{The deviation of this scheme is similar to two-step Adam-Bashforth scheme, thus we name it an AB2-like scheme.} is adopted to estimate the LTE of the DLN method. The AB2-like time-stepping scheme, applying to the initial value problem in \eqref{eq:IVP}, takes the form
	\begin{align}
		y_{n\!+\!1}^{\tt{AB2}\text{-}\tt{like}} \!&=\! y_{n} \!+\! \frac{t_{n\!+\!1} \!-\!t_{n}}{2 (t_{n\!-\!1,\beta} \!-\!t_{n\!-\!2,\beta})} 
		\Big[ (t_{n\!+\!1} \!+\!t_{n} \!-\!2 t_{n\!-\!2,\beta} )  g^{\tt{DLN}} (t_{n\!-\!1,\beta}, y_{n\!-\!1,\beta})   
		\label{eq:AB2-like} \\
		& \qquad \qquad \qquad \qquad \qquad \quad 
		\!-\! (t_{n\!+\!1} \!+\!t_{n} \!-\! 2 t_{n\!-\!1,\beta} ) g^{\tt{DLN}} (t_{n\!-\!2,\beta}, y_{n\!-\!2,\beta})  \Big], \notag 
	\end{align}
	where $g^{\tt{DLN}}(t_{n\!-\!1,\beta}, y_{n\!-\!1,\beta}) $ and $g^{\tt{DLN}} (t_{n\!-\!2,\beta}, y_{n\!-\!2,\beta})$ are calculated by the DLN scheme in \eqref{eq:1legDLN} 
	\begin{align*}
		g^{\tt{DLN}}(t_{n\!-\!1,\beta}, y_{n\!-\!1,\beta}) = \frac{y_{n-1,\alpha}}{\widehat{k}_{n-1}}, \qquad 
		g^{\tt{DLN}}(t_{n\!-\!2,\beta}, y_{n\!-\!2,\beta}) = \frac{y_{n-2,\alpha}}{\widehat{k}_{n-2}}.
	\end{align*}
	Thus the AB2-like solution $y_{n\!+\!1}^{\tt{AB2}\text{-}\tt{like}} $ in \eqref{eq:AB2-like} is just the linear combination of the previous four DLN solutions $\{ y_{n}, y_{n-1}, y_{n-2}, y_{n-3} \}$, 
    and thus obtained with low computational costs. 

    We denote time step ratio $\tau_{n} = k_{n}/k_{n-1}$ and utilize the AB2-like scheme \eqref{eq:AB2-like} to estimate the LTE of the DLN method via
	\begin{gather}
		\widehat{T}_{n+1}
		=\frac{-G^{(n)}}{G^{(n)}+\mathcal{R}^{(n)}} \big( y_{n+1}^{\tt{DLN}} - y_{n+1}^{\tt{AB2}\text{-}\tt{like}} \big),
		\label{eq:Estimator-LTE}
	\end{gather}
	where $G^{(n)}$, $\mathcal{R}^{(n)}$ are coefficients before the LTEs of the DLN scheme and AB2-like scheme respectively: 
	\begin{align}
		G^{(n)} \!=\!& \Big( \frac{1}{2} \!-\! \frac{ \alpha_{0}}{2 \alpha_{2}} \frac{1}{\tau_{n}} \Big)
		\Big( \beta_{2}^{(n)} \!-\! \beta_{0}^{(n)} \frac{1}{\tau_{n}}  \Big)^{2} 
		\!+\! \frac{ \alpha_{0}}{6 \alpha_{2}} \Big( \frac{1}{\tau_{n}} \Big)^{3} \!-\! \frac{1}{6}, 
		\label{eq:EST-LTE-coeffi}\\
		\mathcal{R}^{(n)} \!=\! & \frac{1}{12} \Big[ 2 \!+\! \frac{3}{\tau_n} \Big( 1 \!-\! \beta_2^{(n\!-\!2)} \frac{1}{\tau_{n\!-\!1}} \!+\!\beta_0^{(n\!-\! 2)} \frac{1}{\tau_{n\!-\!2}} \frac{1}{\tau_{n\!-\!1}} \Big) \!
		\Big( 1\!-\!\beta_2^{(n\!-\!1)} \frac{1}{\tau_n}\!+\!\beta_0^{(n\!-\!1)} \frac{1}{\tau_{n\!-\!1}} \frac{1}{\tau_n} \Big) \notag \\
		&\!+\! \frac{3}{\tau_n} \! \Big( 1\!+\!\frac{1}{\tau_n} \!-\!\beta_2^{(n\!-\!2)} \frac{1}{\tau_{n\!-\!1}} \frac{1}{\tau_n}\!+\!\beta_0^{(n\!-\!2)} \frac{1}{\tau_{n\!-\!2}} \frac{1}{\tau_{n\!-\!1}} \frac{1}{\tau_n} \Big)
		\!\Big(\! -\beta_2^{(n\!-\!1)} \!+\! \beta_0^{(n\!-\!1)} \frac{1}{\tau_{n\!-\!1}} \Big) \Big]. \notag 
	\end{align}
    We refer to \cite{LPT23_ACSE} for the derivation of the explicit AB2-like scheme in \eqref{eq:AB2-like} and the estimator of LTE in \eqref{eq:Estimator-LTE}.

    For the time step controller, there are many choices and we consider the following one proposed in \cite{HNW93II}
    \begin{gather}
    	k_{n+1} = k_{n} \cdot \min \Big\{ 1.5, \max \Big\{ 0.2, \kappa \big( {\tt{Tol}}/\| \widehat{T}_{n+1} \|  \big)^{\frac{1}{3}} \Big\} \Big\},
    	\label{eq:improve-controller}
    \end{gather}
    where $\kappa \in (0, 1]$ is the safety factor and ${\tt{Tol}}$ is the required tolerance for the LTE.
    We summarize the adaptive DLN algorithm using the estimator of LTE in \eqref{eq:Estimator-LTE} and the step controller in \eqref{eq:improve-controller} in the following pseudo code.
    \LinesNumberedHidden
    \begin{algorithm}[ptbh]
    	\caption{Adaptive DLN method (LTE estimator by AB2-like scheme)}
    	\label{alg:Adaptivity-AB2-like}
    	\KwIn{tolerance $\text{Tol}$, four previous DLN solutions $u_{n}^{h},u_{n-1}^{h},u_{n-2}^{h},u_{n-3}^{h}$,
    		current time step $k_{n}$, three previous time step $k_{n-1},k_{n-2},k_{n-3}$, safety factor $\kappa$ \;}
    	Compute the DLN solution $u_{n+1}^{h,\tt DLN}$ by \eqref{eq:DLN-CSS-Alg} or \eqref{eq:DLN-SAV-Alg-weak} \;
    	Compute the AB2-like solution $u_{n+1}^{h,\tt{AB2}\text{-}\tt{like}}$ by \eqref{eq:AB2-like} \;
    	Use $k_{n},k_{n-1},k_{n-2},k_{n-3}$ to update $\tau_{n},\tau_{n-1},\tau_{n-2}$ \;
    	Compute $G^{(n)}$, ${\cal R}^{(n)}$ by \eqref{eq:EST-LTE-coeffi} \;
    	$\widehat{T}_{n+1} \Leftarrow \frac{ |G^{(n)}|}{ | G^{(n)} + {\cal R}^{(n)} | } 
    	\| u_{n+1}^{h,\tt DLN}-u_{n+1}^{h,\tt{AB2}\text{-}\tt{like}} \|$  \tcp*{absolute estimator}
    	or $\widehat{T}_{n+1} \Leftarrow \frac{ |G^{(n)}|}{ | G^{(n)} + {\cal R}^{(n)} | } 
    	\frac{\| u_{n+1}^{h,\tt DLN} - u_{n+1}^{h,\tt{AB2}\text{-}\tt{like}} \|}{\| u_{n+1}^{h,\tt DLN} \|}$    \tcp*{relative estimator}
    	\uIf{$ \widehat{T}_{n+1}  < \rm{Tol}$}
    	{
    		$u_{n+1}^{h} \Leftarrow u_{n+1}^{h,\tt DLN}$  \tcp*{accept the result}
    		$k_{n\!+\!1} \!\Leftarrow \! k_{n} \cdot \min \big\{ \!1.5, \max \big\{\!0.2, \kappa \big(\!\frac {\text{Tol}}{ \widehat{T}_{n+1}  }\!\big)^{1/3} \big\} \!\big\}$  
    		 \tcp*{adjust step by \eqref{eq:improve-controller}}
    	}\Else
    	{
    		\tt{// adjust current step to recompute}  
    		$k_{n} \!\Leftarrow \! k_{n} \cdot \min \big\{ 1.5, \max \big\{0.2, \kappa \big(\frac {\text{Tol}}{ \widehat{T}_{n+1}  }\big)^{1/3} \big\} \!\big\}$ \!\;
    	}
    \end{algorithm}

    \section{Numerical Tests}
	\label{sec:NumercalTests}
	In this section, we test the performance of the proposed modified DLN algorithm and DLN-SAV algorithm on three numerical tests: 
	\begin{itemize}
		\item[$\bullet$] Accuracy test on the 1D traveling wave solution;
		\item[$\bullet$] Accuracy test on the 2D test with known solutions;
		\item[$\bullet$] Simulation of the 2D test with random initial conditions.
	\end{itemize}
	We implement these algorithms with three different values of $\theta$: $2/3,2/\sqrt{5}, 1$.
	Among these values, $\theta = 2/3$ is proposed in \cite{DLN83_SIAM_JNA} to balance the magnitude of LTE and keep fine stability properties. 
	The value $\theta = 2/\sqrt{5}$ is suggested in \cite{KS05} to have the best stability at infinity. 
	The algorithm with value $\theta = 1$ reduces to the classical one-step midpoint rule.
	In the implementation, we use MATLAB for the 1D test problem and the software FreeFem++ \cite{Hec12_JNM} for two 2D test problems. For the modified DLN algorithm \eqref{eq:DLN-CSS-Alg}, we use fixed point iteration to solve the non-linear system at each time step.

    \subsection{One-dimensional Traveling Wave Solution}
    \label{sec:NumTest:1D}
	We use the 1D traveling wave problem \cite{CLJK09_PhyA} with the known solution to test the accuracy of modified DLN and DLN-SAV algorithms.
	One traveling wave solution of Allen-Cahn equation \eqref{eq:AllenCahn_Eq} in 1D takes the form
	\begin{gather*}
		u(x,t) = \frac{1}{2} \Big[ 1 - \tanh\Big( \frac{x-st}{2\sqrt{2} \epsilon}   \Big)   \Big], \qquad
		-2 \leq x \leq 4,
	\end{gather*}
	where the traveling speed is $s = 3 \epsilon /\sqrt{2}$. 
	We set the final time as $T =2$, the model parameter as $\epsilon = 0.01$, and use the inhomogeneous Dirichlet boundary condition. We use $\mathbb{P}^2$ finite element space in the spatial discretization.

    \subsubsection{Constant Time Step}
	First, we consider the constant time step $k$ and set $h = {k}^{2}$ to check the accuracy of both algorithms in time.
	From Tables \ref{table:DLN-CSS-1D-dt} and \ref{table:DLN-SAV-1D-dt}, we observe that both modified DLN and DLN-SAV algorithms have second-order accuracy in time for the 1D traveling wave problem, confirming the expected time-stepping accuracy of these algorithms.
	Next, we set $k = h^{2}$ to check the accuracy of both algorithms in space. 
    From Tables \ref{table:DLN-CSS-1D-h} and \ref{table:DLN-SAV-1D-h},  we can see that both algorithms demonstrate third-order spatial convergence for $\ell^{\infty}({L^{2}})$-norm and second-order spatial convergence for $\ell^{2}(H^{1})$-norm.
    \begin{table}[ptbh]
        \centering
        \renewcommand\arraystretch{1.25}
        \caption{$L^{2}$-norm of error and rate for 1D modified DLN scheme in time ($h \!=\! {k} ^2$) } 
        \begin{tabular}{ccccccccc}
          \hline
          \hline
          & $\!k \!$  & $\!\| u \!-\! u^{h} \|_{\ell^{\infty}(L^{2})}\!$ & $\!R\!$
          & $\!\| u \!-\! u^{h} \|_{\ell^{2}(L^{2})}\!$ & $\!R\!$ 
          & $\!\| u \!-\! u^{h} \|_{\ell^{2}(H^{1})}\!$ & $\!R\!$
          \Tstrut
          \\
          \hline 
          \hline 
          $\!\!\!\!\!\theta \!=\! \frac{2}{3}$  
          & $0.04$     & 1.84e-5   & -   & 1.56e-5   & -   & 1.05e-3    & -
          \\
          & $0.02$     & 4.64e-6   & 1.98   & 3.92e-6   & 1.99   & 1.58e-4    & 2.73 
          \\
          & $0.01$     & 1.17e-6   & 1.99   & 9.82e-7   & 2.00   & 3.74e-5    & 2.08 
          \\
          & $0.005$     & 2.92e-7   & 2.00   & 2.46e-7   & 2.00   & 9.30e-6    & 2.01 
          \\
          \hline
          \hline 
          $\!\!\!\!\!\theta \!=\! \frac{2}{\sqrt{5}}$ 
          & $0.04$     & 1.45e-5   & -   & 1.14e-5   & -   & 9.38e-4    & -
          \\
          & $0.02$     & 3.68e-6   & 1.98   & 2.88e-6   & 1.99   & 1.05e-4    & 3.16
          \\
          & $0.01$     & 9.27e-7   & 1.99   & 7.23e-7   & 1.99   & 2.27e-5    & 2.21
          \\
          & $0.005$     & 2.32e-7   & 2.00   & 1.81e-7   & 2.00   & 5.61e-6    & 2.02
          \\
          \hline
          \hline 
          $\!\!\!\!\!\theta \!=\! 1$ 
          & $0.04$     & 1.27e-5     & -   & 9.58e-6   & -     & 9.19e-4    & -
          \\
          & $0.02$     & 3.22e-6     & 1.98   & 2.42e-6   & 1.99     & 9.45e-5    & 3.28
          \\
          & $0.01$     & 8.12e-7     & 1.99   & 6.07e-7   & 2.00     & 1.97e-5    & 2.26
          \\
          & $0.005$     & 2.04e-7   & 1.99   & 1.52e-7   & 2.00   & 4.86e-6    & 2.02
          \\
          \hline
          \hline
        \end{tabular}
        \label{table:DLN-CSS-1D-dt}
    \end{table}

    \begin{table}[ptbh]
		\centering
		\renewcommand\arraystretch{1.25}
		\caption{$L^{2}$-norm of error and rate for 1D modified DLN scheme in space (${k} \!=\! h^2$)} 
		\begin{tabular}{ccccccccc}
			\hline
			\hline
			& $\!h\! $  & $\!\| u \!-\! u^{h} \|_{\ell^{\infty}(L^{2})}\!$ & $\!R\!$
			& $\!\| u \!-\! u^{h} \|_{\ell^{2}(L^{2})}\!$ & $\!R\!$ 
			& $\!\| u \!-\! u^{h} \|_{\ell^{2}(H^{1})}\!$ & $\!R\!$
			\Tstrut
			\\
			\hline
			\hline 
			$\!\!\!\!\! \theta \!=\! \frac{2}{3}$ & $0.04$       & 2.17e-3   & -         & 2.41e-3   & -        & 4.80e-1   & -
			\Tstrut
			\\
			& $0.02$       & 2.91e-4   & 2.90    & 3.97e-4   & 2.60   & 1.32e-1   & 1.86
			\\
			& $0.01$       & 3.69e-5   & 2.98    & 5.19e-5   & 2.94   & 3.39e-2   & 1.97
			\\
			& $0.005$     & 4.65e-6   & 2.99    & 6.57e-6   & 2.98   & 8.53e-3   & 1.99
			\\
			\hline
			\hline 
			$\!\!\!\!\! \theta \!=\! \frac{2}{\sqrt{5}}$ & $0.04$       & 2.17e-3    & -         & 2.41e-3   & -        & 4.80e-1   & -
			\Tstrut
			\\
			& $0.02$       & 2.91e-4    & 2.90    & 3.97e-4   & 2.60   & 1.32e-1   & 1.86
			\\
			& $0.01$       & 3.69e-5    & 2.98    & 5.19e-5   & 2.94   & 3.39e-2   & 1.97
			\\
			& $0.005$     & 4.65e-6    & 2.99    & 6.57e-6   & 2.98   & 8.53e-3   & 1.99
			\\
			\hline
			\hline 
			$\!\!\!\!\! \theta \!=\! 1$ & $0.04$       & 2.17e-3    & -         & 2.41e-3   & -        & 4.80e-1   & -
			\\
			& $0.02$       & 2.91e-4    & 2.90    & 3.97e-4   & 2.60   & 1.32e-1   & 1.86
			\\
			& $0.01$       & 3.69e-5    & 2.98    & 5.19e-5   & 2.94   & 3.39e-2   & 1.97
			\\
			& $0.005$     & 4.65e-6    & 2.99    & 6.57e-6   & 2.98   & 8.53e-3   & 1.99
			\\
			\hline
			\hline
		\end{tabular}
		\label{table:DLN-CSS-1D-h}
	\end{table}

    \begin{table}[ptbh]
		\centering
		\renewcommand\arraystretch{1.25}
		\caption{$L^{2}$-norm of error and rate for 1D DLN-SAV scheme in time ($h \!=\! {k} ^2$)} 
		\begin{tabular}{ccccccccc}
			\hline
			\hline
			& $\!k \!$  & $\!\| u \!-\! u^{h} \|_{\ell^{\infty}(L^{2})}\!$ & $\!R\!$
			& $\!\| u \!-\! u^{h} \|_{\ell^{2}(L^{2})}\!$ & $\!R\!$ 
			& $\!\| u \!-\! u^{h} \|_{\ell^{2}(H^{1})}\!$ & $\!R\!$
			\Tstrut
			\\
			\hline
			\hline
			$\!\!\!\!\!\theta \!=\! \frac{2}{3}$ & $0.04$       & 4.94e-5   & -          & 3.82e-5   & -       & 2.09e-3   & -
			\Tstrut
			\\
			& $0.02$     & 1.26e-5   & 1.98     & 9.64e-6   & 1.98  & 4.82e-4   & 2.12
			\\
			& $0.01$     & 3.17e-6   & 1.99     & 2.42e-6   & 1.99  & 1.20e-4   & 2.00
			\\
			& $0.005$     & 7.94e-7   & 1.99     & 6.07e-7   & 2.00  & 3.01e-5   & 2.00
			\\
			\hline
			\hline
			$\!\!\!\!\!\theta \!=\! \frac{2}{\sqrt{5}}$ & $0.04$      & 6.05e-5    & -        & 4.61e-5   & -        & 2.07e-3   & -
			\Tstrut
			\\
			& $0.02$     & 1.54e-5   & 1.97   & 1.17e-5   & 1.98   & 4.77e-4   & 2.11
			\\
			& $0.01$     & 3.90e-6   & 1.99   & 2.94e-6   & 1.99   & 1.19e-4   & 2.00
			\\
			& $0.005$     & 9.79e-7   & 1.99   & 7.38e-7   & 2.00   & 2.99e-5   & 2.00
			\\
			\hline
			\hline
			$\!\!\!\!\!\theta \!=\! 1$ & $0.04$       & 6.61e-5   & -         & 5.03e-5   & -        & 2.05e-3    & -
			\Tstrut
			\\
			& $0.02$     & 1.69e-5   & 1.97    & 1.28e-5   & 1.98   & 4.74e-4    & 2.11
			\\
			& $0.001$     & 4.26e-6   & 1.99    & 3.21e-6   & 1.99   & 1.18e-4    & 2.00
			\\
			& $0.005$     & 1.07e-6   & 1.99    & 8.06e-7   & 1.99   & 2.97e-5    & 2.00
			\\
			\hline
			\hline
		\end{tabular}
		\label{table:DLN-SAV-1D-dt}
	\end{table}

    \begin{table}[ptbh]
		\centering
		\renewcommand\arraystretch{1.25}
		\caption{$L^{2}$-norm of error and rate for 1D DLN-SAV scheme in space (${k} \!=\! h^2$)} 
		\begin{tabular}{cccccccc}
			\hline
			\hline
			& $\!h\! $  & $\!\| u \!-\! u^{h} \|_{\ell^{\infty}(L^{2})}\!$ & $\!R\!$
			& $\!\| u \!-\! u^{h} \|_{\ell^{2}(L^{2})}\!$ & $\!R\!$ 
			& $\!\| u \!-\! u^{h} \|_{\ell^{2}(H^{1})}\!$ & $\!R\!$
			\Tstrut
			\\
			\hline
			\hline
			$\!\!\!\!\!\theta \!=\! \frac{2}{3}$ & $0.04$      & 2.17e-3    & -         & 2.41e-3   & -        & 4.80e-1   & -
			\Tstrut
			\\
			& $0.02$      & 2.91e-4    & 2.90    & 3.97e-4   & 2.60   & 1.32e-1   & 1.86
			\\
			& $0.01$      & 3.69e-5    & 2.98    & 5.19e-5   & 2.94   & 3.39e-2   & 1.97
			\\
			& $0.005$     & 4.65e-6   & 2.99    & 6.57e-6   & 2.98   & 8.53e-3   & 1.99
			\\
			\hline
			\hline
			$\!\!\!\!\! \theta \!=\! \frac{2}{\sqrt{5}}$ & $0.04$       & 2.17e-3   & -         & 2.41e-3   & -        & 4.80e-1   & -
			\Tstrut
			\\
			& $0.02$       & 2.91e-4   & 2.90    & 3.97e-4   & 2.60   & 1.32e-1   & 1.86
			\\
			& $0.01$       & 3.69e-5   & 2.98    & 5.19e-5   & 2.94   & 3.39e-2   & 1.97
			\\
			& $0.005$     & 4.65e-6   & 2.99    & 6.57e-6   & 2.98   & 8.53e-3   & 1.99
			\\
			\hline
			\hline
			$\!\!\!\!\! \theta \!=\! 1$ & $0.04$       & 2.17e-3   & -         & 2.41e-3   & -        & 4.80e-1   & -
			\Tstrut
			\\
			& $0.02$       & 2.91e-4   & 2.90    & 3.97e-4   & 2.60   & 1.32e-1   & 1.86
			\\
			& $0.01$       & 3.69e-5   & 2.98    & 5.19e-5   & 2.94   & 3.39e-2   & 1.97
			\\
			& $0.005$     & 4.65e-6   & 2.99    & 6.57e-6   & 2.98   & 8.53e-3   & 1.99
			\\
			\hline
			\hline
		\end{tabular}
		\label{table:DLN-SAV-1D-h}
	\end{table}

    \subsubsection{Variable Time Step}
    We utilize constant time step $k$ as a reference time step and set $h = k^2$ to test the accuracy of the variable time-stepping modified DLN algorithm \eqref{eq:DLN-CSS-Alg} and DLN-SAV algorithm \eqref{eq:DLN-SAV-Alg-weak} over time. Two different time step scenarios are considered:
    \begin{enumerate}
            \item Randomized Time Steps: For each time step, we set $k_n = k + k \cdot rand$, where $rand$ is a random number drawn from the uniform distribution in the interval (0,1). 
            \item Alternating Time Steps: We alternate the time step size as follows: $k_1 = k, k_2 = 2k, k_3 = k, k_4 = 2k, ...$
    \end{enumerate}
    For both test scenarios, the time steps lie within the range $[k, 2k]$. The accuracy rate is calculated using the formula:
    \begin{equation*}
        \text{Rate} = \frac{\ln(\text{Error}_1/\text{Error}_2)}{\ln(k_{\text{max},1}/k_{\text{max},2})}
    \end{equation*}
    where \(\text{Error}_1\) and \(\text{Error}_2\) are the errors corresponding to the maximum time steps \(k_{\text{max},1}\) and \(k_{\text{max},2}\), respectively.

    From Tables \ref{table:DLN-CSS-1D-random-t} - \ref{table:DLN-SAV-1D-alternating-t}, our observation indicates that both algorithms with all $\theta$ values achieve second-order accuracy in time for the 1D traveling wave problem, even under non-uniform time steps. This is better than the first order convergence in time that we can prove in Theorem \ref{theorem2} for the non-uniform time step case. 

    To further test the order of accuracy of the modified DLN scheme, we apply the modified DLN method on the corresponding ODE ${dy}/{dt} + f(y) = 0$, assuming $y=y(t)$ is only time dependent.
    We have carried out extensive numerical investigation of this test with different ways to determine the non-uniform time step size and different choices of $f(u)$.
    All of these tests demonstrate a second-order convergence numerically.

    \begin{table}[ptbh]
		\centering
		\renewcommand\arraystretch{1.25}
		\caption{$L^{2}$-norm of error and rate for 1D modified DLN scheme in time ($k_n = k + k \cdot rand; h \!=\! k^2$)} 
		\begin{tabular}{ccccccccc}
			\hline
			\hline
			& $k$ & $\!k_{\text{max}}\! $  & $\!\| u \!-\! u^{h} \|_{\ell^{\infty}(L^{2})}\!$ & $\!R\!$
			& $\!\| u \!-\! u^{h} \|_{\ell^{2}(L^{2})}\!$ & $\!R\!$ 
			& $\!\| u \!-\! u^{h} \|_{\ell^{2}(H^{1})}\!$ & $\!R\!$
			\Tstrut
			\\
			\hline
			\hline
			$\!\!\!\!\!\theta \!=\! \frac{2}{3}$ & $0.1$ & 0.1919  & 2.20e-4   & -    & 2.06e-4    & -         & 3.45e-2   & - 
			\Tstrut
			\\
			& $0.05$ & 0.0991  & 5.98e-5   & 1.97   & 5.33e-5   & 2.04    & 3.45e-3   & 3.48   
			\\
			& $0.04$ & 0.0771  & 3.36e-5   & 2.31    & 3.02e-5   & 2.28   & 1.80e-3   & 2.61
			\\
			& $0.02$ & 0.0399 & 1.01e-5   & 1.82    & 8.58e-6   & 1.91   & 4.50e-4   & 2.10
			\\
			\hline
			\hline
			$\!\!\!\!\! \theta \!=\! \frac{2}{\sqrt{5}}$ & $0.1$ & 0.1995  & 1.85e-4   & -         & 1.48e-4   & -        & 3.34e-2  & -
			\Tstrut
			\\
			& $0.05$ & 0.0992 & 5.61e-5   & 1.71    & 4.57e-5   & 1.68   & 2.71e-3   & 3.60
			\\
			& $0.04$ & 0.0789  & 3.56e-5   & 1.99    & 2.72e-5   & 2.26   & 1.32e-3   & 3.13
			\\
			& $0.02$ & 0.0398  & 8.16e-6   & 2.15    & 6.37e-6   & 2.12   & 2.37e-4   & 2.51
			\\
			\hline
			\hline
			$\!\!\!\!\! \theta \!=\! 1$ & $0.1$  & 0.1738  & 1.53e-4   & -         & 1.30e-4   & -        & 3.32e-2   & -
			\Tstrut
			\\
			& $0.05$ & 0.0994  & 4.81e-5   & 2.07    & 3.55e-5   & 2.32   & 2.40e-3   & 4.70
			\\
			& $0.04$ & 0.0787  & 3.01e-5  & 2.00    & 2.28e-5   & 1.91   & 1.15e-3   & 3.22
			\\
			& $0.02$ & 0.0391  & 7.86e-6   & 1.92    & 6.03e-6  & 1.90   & 2.02e-4   & 2.47
			\\
			\hline
			\hline
		\end{tabular}
		\label{table:DLN-CSS-1D-random-t}
	\end{table}

    \begin{table}[ptbh]
		\centering
		\renewcommand\arraystretch{1.25}
		\caption{$L^{2}$-norm of error and rate for 1D modified DLN scheme in time ($k_n = k, 2k, k, 2k,...; h \!=\! k^2$)} 
		\begin{tabular}{ccccccccc}
			\hline
			\hline
			& $k$ & $\!k_{max}\! $  & $\!\| u \!-\! u^{h} \|_{\ell^{\infty}(L^{2})}\!$ & $\!R\!$
			& $\!\| u \!-\! u^{h} \|_{\ell^{2}(L^{2})}\!$ & $\!R\!$ 
			& $\!\| u \!-\! u^{h} \|_{\ell^{2}(H^{1})}\!$ & $\!R\!$
			\Tstrut
			\\
			\hline
			\hline
			$\!\!\!\!\!\theta \!=\! \frac{2}{3}$ & $0.1$ & 0.2 & 2.32e-4   & -    & 2.51e-4    & -         & 3.56e-2   & - 
			\Tstrut
			\\
			& $0.05$ & 0.1  & 6.20e-5   & 1.90   & 6.26e-5   & 2.00    & 4.04e-3   & 3.14   
			\\
			& $0.04$ & 0.08  & 3.89e-5   & 2.08    & 4.00e-5   & 2.01   & 2.37e-3   & 2.40
			\\
			& $0.02$ & 0.04 & 1.01e-5   & 1.95    & 1.00e-5   & 1.99   & 5.60e-4   & 2.08
			\\
			\hline
			\hline
			$\!\!\!\!\! \theta \!=\! \frac{2}{\sqrt{5}}$ & $0.1$ & 0.2  & 2.28e-4   & -         & 1.96e-4   & -        & 3.37e-2  & -
			\Tstrut
			\\
			& $0.05$ & 0.1 & 5.97e-5   & 1.93    & 4.75e-5   & 2.05   & 2.76e-3   & 3.61
			\\
			& $0.04$ & 0.08  & 3.79e-5   & 2.03    & 3.03e-5   & 2.01   & 1.43e-3   & 2.95
			\\
			& $0.02$ & 0.04  & 9.71e-6  & 1.97    & 7.54e-6   & 2.01   & 2.86e-4   & 2.32
			\\
			\hline
			\hline
			$\!\!\!\!\! \theta \!=\! 1$ & $0.1$  & 0.2  & 2.28e-4   & -         & 1.91e-4   & -        & 3.34e-2   & -
			\Tstrut
			\\
			& $0.05$ & 0.1  & 5.95e-5   & 1.93    & 4.62e-5  & 2.05   & 2.57e-3   & 3.70
			\\
			& $0.04$ & 0.08  & 3.80e-5  & 2.02    & 2.95e-5   & 2.02  & 1.28e-3  & 3.13
			\\
			& $0.02$ & 0.04  & 9.70e-6   & 1.97    & 7.35e-6  & 2.00  & 2.41e-4  & 2.41
			\\
			\hline
			\hline
		\end{tabular}
		\label{table:DLN-CSS-1D-alternating-t}
	\end{table}

    \begin{table}[ptbh]
		\centering
		\renewcommand\arraystretch{1.25}
		\caption{$L^{2}$-norm of error and rate for 1D DLN-SAV scheme in time ($k_n = k + k \cdot rand; h \!=\! k^2$)} 
		\begin{tabular}{ccccccccc}
			\hline
			\hline
			& $k$ & $\!k_{\text{max}}\! $  & $\!\| u \!-\! u^{h} \|_{\ell^{\infty}(L^{2})}\!$ & $\!R\!$
			& $\!\| u \!-\! u^{h} \|_{\ell^{2}(L^{2})}\!$ & $\!R\!$ 
			& $\!\| u \!-\! u^{h} \|_{\ell^{2}(H^{1})}\!$ & $\!R\!$
			\Tstrut
			\\
			\hline
			\hline
			$\!\!\!\!\!\theta \!=\! \frac{2}{3}$ & $0.1$ & 0.1885  & 6.25e-4   & -    & 4.85e-4    & -         & 4.16e-2   & - 
			\Tstrut
			\\
			& $0.05$ & 0.1000  & 1.62e-4   & 2.13   & 1.31e-4   & 2.07    & 6.75e-3   & 2.87   
			\\
			& $0.04$ & 0.0781  & 1.20e-4   & 1.20    & 9.35e-5   & 1.36   & 4.80e-3   & 1.37
			\\
			& $0.02$ & 0.0399 & 3.10e-5   & 2.02    & 2.45e-5   & 1.99   & 1.22e-3   & 2.04
			\\
			\hline
			\hline
			$\!\!\!\!\! \theta \!=\! \frac{2}{\sqrt{5}}$ & $0.1$ & 0.1931  & 7.12e-4   & -   & 5.64e-4   & -        & 4.10e-2  & -
			\Tstrut
			\\
			& $0.05$ & 0.0987 & 2.20e-4   & 1.76    & 1.66e-4   & 1.82   & 7.20e-3   & 2.59
			\\
			& $0.04$ & 0.0744  & 1.28e-4   & 1.89    & 9.68e-5   & 1.91   & 4.05e-3   & 2.03
			\\
			& $0.02$ & 0.0400  & 3.74e-5   & 1.98    & 2.80e-5   & 2.00   & 1.15e-3   & 2.03
			\\
			\hline
			\hline
			$\!\!\!\!\! \theta \!=\! 1$ & $0.1$  & 0.1996  & 8.19e-4   & -   & 6.41e-4   & -     & 4.13e-2   & -
			\Tstrut
			\\
			& $0.05$ & 0.0998  & 2.58e-4   & 1.67    & 2.04e-4   & 1.65   & 7.94e-3   & 2.38
			\\
			& $0.04$ & 0.0770  & 1.38e-4  & 2.42    & 1.05e-4   & 2.57   & 4.04e-3   & 2.62
			\\
			& $0.02$ & 0.0396  & 3.86e-5   & 1.91    & 2.96e-5  & 1.90   & 1.11e-3   & 1.93
			\\
			\hline
			\hline
		\end{tabular}
		\label{table:DLN-SAV-1D-random-t}
	\end{table}

    \begin{table}[ptbh]
		\centering
		\renewcommand\arraystretch{1.25}
		\caption{$L^{2}$-norm of error and rate for 1D DLN-SAV scheme in time ($k_n = k, 2k, k, 2k,...; h \!=\! k^2$)} 
		\begin{tabular}{ccccccccc}
			\hline
			\hline
			& $k$ & $\!k_{\text{max}}\! $  & $\!\| u \!-\! u^{h} \|_{\ell^{\infty}(L^{2})}\!$ & $\!R\!$
			& $\!\| u \!-\! u^{h} \|_{\ell^{2}(L^{2})}\!$ & $\!R\!$ 
			& $\!\| u \!-\! u^{h} \|_{\ell^{2}(H^{1})}\!$ & $\!R\!$
			\Tstrut
			\\
			\hline
			\hline
			$\!\!\!\!\!\theta \!=\! \frac{2}{3}$ & $0.1$ & 0.2  & 6.52e-4   & -    & 5.53e-4    & -         & 4.37e-2   & - 
			\Tstrut
			\\
			& $0.05$ & 0.1  & 1.76e-4   & 1.89   & 1.41e-4   & 1.97    & 7.41e-3   & 2.56   
			\\
			& $0.04$ & 0.08  & 1.14e-4   & 1.96    & 9.07e-5   & 1.99   & 4.64e-3   & 2.10
			\\
			& $0.02$ & 0.04 & 2.94e-5   & 1.95    & 2.29e-5   & 1.99   & 1.15e-3   & 2.01
			\\
			\hline
			\hline
			$\!\!\!\!\! \theta \!=\! \frac{2}{\sqrt{5}}$ & $0.1$ & 0.2  & 7.95e-4   & -   & 6.57e-4   & -        & 4.36e-2  & -
			\Tstrut
			\\
			& $0.05$ & 0.1 & 2.18e-4   & 1.86    & 1.71e-4   & 1.94   & 7.47e-3   & 2.55
			\\
			& $0.04$ & 0.08  & 1.40e-4   & 1.97    & 1.10e-4   & 1.98   & 4.68e-3   & 2.09
			\\
			& $0.02$ & 0.04  & 3.63e-5   & 1.95    & 2.78e-5   & 1.98   & 1.16e-3   & 2.01
			\\
			\hline
			\hline
			$\!\!\!\!\! \theta \!=\! 1$ & $0.1$  & 0.2  & 8.66e-4   & -   & 7.10e-4   & -     & 4.35e-2   & -
			\Tstrut
			\\
			& $0.05$ & 0.1  & 2.38e-4   & 1.86    & 1.86e-4   & 1.93   & 7.45e-3   & 2.54
			\\
			& $0.04$ & 0.08  & 1.54e-4  & 1.97    & 1.20e-4   & 1.97   & 4.67e-3   & 2.09
			\\
			& $0.02$ & 0.04  & 3.98e-5   & 1.95    & 3.04e-5  & 1.98   & 1.16e-3   & 2.01
			\\
			\hline
			\hline
		\end{tabular}
		\label{table:DLN-SAV-1D-alternating-t}
	\end{table}

    \subsubsection{Adaptive Algorithms}
	Next, we test the accuracy and efficiency of the adaptive DLN algorithm in Algorithm \ref{alg:Adaptivity-AB2-like}, by comparing it with the corresponding constant time step algorithms. For adaptive algorithms, the following setup is adopted: the tolerance $\rm{Tol} = 1.\rm{e}-6$, the minimum time step $k_{\rm{min}} = 1.\rm{e}-5$, the maximum time step $k_{\rm{max}} = 0.1$, two initial time steps $k_{0} = k_{1} = 1.\rm{e}-3$, and the safety factor $\kappa = 0.8$. 
    For constant step algorithms, we set the constant time step as $k = 1.\rm{e}-3$. 
    From Tables \ref{table:Adapt-Const-DLN-CSS-1D} and \ref{table:Adapt-Const-DLN-SAV-1D}, we observe that both adaptive algorithms outperform the constant step algorithms. Adaptive algorithms take much fewer number of time steps to obtain the same accuracy for all three $\theta$ values, when compared with the corresponding constant step algorithms.
    \begin{table}[ptbh]
		\centering
		\renewcommand\arraystretch{1.25}
		\caption{Errors and number of time steps of the adaptive and constant time-stepping 1D modified DLN schemes} 
		\begin{tabular}{cccccc}
			\hline
			\hline
			\multicolumn{6}{c}{Adaptive modified DLN schemes}          
			\Tstrut \\
			\hline
			$\!\!\!\!\theta \!$  & $\!\| u \!-\! u^{h} \|_{\ell^{\infty}(L^{2})} \!$  & $\!\| u \!-\! u^{h} \|_{\ell^{2}(L^{2})}\!$ 
			& $\!\| u \!-\! u^{h} \|_{\ell^{2}(H^{1})}\!$ & \# Steps & \# Rejections
			\Tstrut
			\\
			\hline 
			$\!\!\!\!\!\theta \!=\! \frac{2}{3}$                 & 3.69e-5      & 3.67e-5      & 2.39e-2        & 653     &  341
			\Tstrut
			\\
			$\!\!\!\!\! \theta \!=\! \frac{2}{\sqrt{5}}$      & 3.69e-5      & 3.67e-5      & 2.39e-2        & 264     &  93
			\\
			$\!\!\!\!\! \theta \!=\! 1$                              & 3.69e-5      & 3.67e-5      & 2.39e-2        & 148     &  43
			\\
			\hline
			\hline
			\multicolumn{6}{c}{Constant time-stepping modified DLN schemes}          
			\Tstrut \\
			\hline
			$\!\!\!\!\theta \!$  & $\!\| u \!-\! u^{h} \|_{\ell^{\infty}(L^{2})} \!$  & $\!\| u \!-\! u^{h} \|_{\ell^{2}(L^{2})}\!$ 
			& $\!\| u \!-\! u^{h} \|_{\ell^{2}(H^{1})}\!$ & \# Steps & \# Rejections
			\Tstrut
			\\
			\hline
			$\!\!\!\!\!\theta \!=\! \frac{2}{3}$                  & 3.69e-5      & 3.67e-5      & 2.39e-2       & 1000    & -
			\\
			$\!\!\!\!\! \theta \!=\! \frac{2}{\sqrt{5}}$       & 3.69e-5      & 3.67e-5      & 2.39e-2       & 1000    & -
			\\
			$\!\!\!\!\! \theta \!=\! 1$                               & 3.69e-5      & 3.67e-5      & 2.39e-2       & 1000    & - 
			\\
			\hline
		\end{tabular}
		\label{table:Adapt-Const-DLN-CSS-1D}
	\end{table}

    \begin{table}[ptbh]
		\centering
		\renewcommand\arraystretch{1.25}
		\caption{Errors and number of time steps of the adaptive and constant time-stepping 1D DLN-SAV schemes} 
		\begin{tabular}{cccccc}
			\hline
			\hline
			\multicolumn{6}{c}{Adaptive DLN-SAV schemes}          
			\Tstrut \\
			\hline
			$\!\!\!\!\theta \!$  & $\!\| u \!-\! u^{h} \|_{\ell^{\infty}(L^{2})} \!$  & $\!\| u \!-\! u^{h} \|_{\ell^{2}(L^{2})}\!$ 
			& $\!\| u \!-\! u^{h} \|_{\ell^{2}(H^{1})}\!$ & \# Steps & \# Rejections
			\Tstrut
			\\
			\hline
			$\!\!\!\!\!\theta \!=\! \frac{2}{3}$                  & 3.68e-5     & 3.67e-5      & 2.39e-2        & 174    & 65
			\\
			$\!\!\!\!\! \theta \!=\! \frac{2}{\sqrt{5}}$       & 3.67e-5     & 3.67e-5      & 2.39e-2        & 146    & 36
			\\
			$\!\!\!\!\! \theta \!=\! 1$                               & 3.67e-5     & 3.67e-5      & 2.39e-2        & 140    & 43
			\\
			\hline
			\hline
			\multicolumn{6}{c}{Constant time-stepping DLN-SAV schemes}          
			\Tstrut \\
			\hline
			$\!\!\!\!\theta \!$  & $\!\| u \!-\! u^{h} \|_{\ell^{\infty}(L^{2})} \!$  & $\!\| u \!-\! u^{h} \|_{\ell^{2}(L^{2})}\!$ 
			& $\!\| u \!-\! u^{h} \|_{\ell^{2}(H^{1})}\!$ & \# Steps & \# Rejections
			\Tstrut
			\\
			\hline
			$\!\!\!\!\!\theta \!=\! \frac{2}{3}$                 & 3.68e-5     & 3.67e-5      & 2.39e-2     & 1000    & -
			\Tstrut
			\\
			$\!\!\!\!\! \theta \!=\! \frac{2}{\sqrt{5}}$      & 3.68e-5     & 3.67e-5      & 2.39e-2     & 1000    & -
			\\
			$\!\!\!\!\! \theta \!=\! 1$                              & 3.68e-5     & 3.67e-5      & 2.39e-2     & 1000    & -
			\\
			\hline
		\end{tabular}
		\label{table:Adapt-Const-DLN-SAV-1D}
	\end{table}

    \subsection{Two-dimensional Example with Known Solution} 
        We consider a two-dimensional example to verify that the proposed DLN algorithms are accurate and efficient for 2D test problems. 
	As studied in \cite{LHY19_JCAM}, we consider the exact function of the form
	\begin{gather}
		u(x,y,t) = 0.05 e^{-0.1t} \sin(x) \sin(y), \  (x,y) \in  [0,2\pi]\times [0,2\pi], \ 0 \leq t \leq T,
		\label{eq:exact-2D}
	\end{gather}
        with the homogenous boundary condition.
	Since this function does not satisfy the Allen-Cahn equation \eqref{eq:AllenCahn_Eq} exactly, we add the extra source term 
	\begin{gather*}
		g(x,y,t) = 0.05(2\epsilon^2 -1.1) e^{-0.1t} \sin(x) \sin(y) + (0.05 e^{-0.1t} \sin(x) \sin(y))^3
	\end{gather*}
	such that $u(x,y,t)$ in \eqref{eq:exact-2D} is the exact solution to the revised Allen-Cahn equation 
	\begin{gather*}
		u_t - \epsilon^2 \Delta u + f(u) = g(x,y,t).
	\end{gather*}
	with the model parameter $\epsilon = 0.01$. $\mathbb{P}^2$ finite element space is again used.

    \subsubsection{Constant Time Step}
	To confirm the time convergence rate, we first set the final time as $T = 4$ and consider a refined triangular mesh with $N=500$ nodes on each side of the boundaries. 
    Tables \ref{table:DLN-CSS-2D-dt} and \ref{table:DLN-SAV-2D-dt} show that both modified DLN and DLN-SAV algorithms achieved second-order convergence in time for the 2D test problem in \eqref{eq:exact-2D}.
	Then we set the final time $T=1$ and the constant time step $k = 0.01$ to check the spatial convergence. We observe from Tables \ref{table:DLN-CSS-2D-h} and \ref{table:DLN-SAV-2D-h} that both algorithms have third-order accuracy in $L^{2}$-norm and second-order accuracy in $H^{1}$-norm.

    \begin{table}[ptbh]
		\centering
		\renewcommand\arraystretch{1.25}
		\caption{$L^{2}$-norm of error and rate for 2D modified DLN scheme in time ($N = 500$)} 
		\begin{tabular}{ccccccccc}
			\hline
			\hline
			& $\!k \!$  & $\!\| u \!-\! u^{h} \|_{\ell^{\infty}(L^{2})}\!$ & $\!R\!$
			& $\!\| u \!-\! u^{h} \|_{\ell^{2}(L^{2})}\!$ & $\!R\!$ 
			& $\!\| u \!-\! u^{h} \|_{\ell^{2}(H^{1})}\!$ & $\!R\!$
			\Tstrut
			\\
			\hline 
			\hline 
			$\!\!\!\!\!\theta \!=\! \frac{2}{3}$ & $0.4$    & 1.80e-3  & -  &  1.51e-3  & - &  2.13e-3  & -
			\Tstrut \\
			& $0.2$      & 5.25e-4  & 1.78  &  4.03e-4  & 1.90  & 5.70e-4  & 1.90
			\\
			& $0.1$      & 1.43e-4  & 1.88  &  1.05e-4  & 1.95  & 1.48e-4  & 1.95
			\\
			& $0.05$    & 3.72e-5  & 1.94  &  2.66e-5  & 1.97  & 3.79e-5  & 1.97
			\\
			\hline
			\hline 
			$\!\!\!\!\!\theta \!=\! \frac{2}{\sqrt{5}}$ & $0.4$    & 1.31e-3   & -   &  1.10e-3   & -  &   1.55e-3  &  -
			\Tstrut \\
			& $0.2$    & 3.90e-4  & 1.75   & 3.00e-4  & 1.87  &  4.24e-4  &  1.87
			\\
			& $0.1$    & 1.07e-4  & 1.86   & 7.87e-5  & 1.93  &  1.11e-4  &  1.93
			\\
			& $0.05$  & 2.82e-5  & 1.93   & 2.02e-5  & 1.96  &  2.87e-5  &  1.95
			\\
			\hline
			\hline 
			$\!\!\!\!\!\theta \!=\! 1$ & $0.4$    & 1.06e-3  & -  & 8.84e-4   & - &  1.25e-3  & -
			\Tstrut \\
			& $0.2$     &  3.18e-4  &  1.73   & 2.44e-4   & 1.86   &  3.45e-4  &  1.86
			\\
			& $0.1$     & 8.79e-5   &  1.85   & 6.44e-5   & 1.92   &  9.12e-5  &  1.92
			\\
			& $0.05$   & 2.31e-5   &  1.93   & 1.65e-5   & 1.96   &  2.37e-5  &  1.94
			\\
			\hline
			\hline
		\end{tabular}
		\label{table:DLN-CSS-2D-dt}
	\end{table}

    \begin{table}[ptbh]
		\centering
		\renewcommand\arraystretch{1.25}
		\caption{$L^{2}$-norm of error and rate for 2D modified DLN scheme in space ($k=0.01$)} 
		\begin{tabular}{ccccccccc}
			\hline
			\hline
			& $\!N\! $  & $\!\| u \!-\! u^{h} \|_{\ell^{\infty}(L^{2})}\!$ & $\!R\!$
			& $\!\| u \!-\! u^{h} \|_{\ell^{2}(L^{2})}\!$ & $\!R\!$ 
			& $\!\| u \!-\! u^{h} \|_{\ell^{2}(H^{1})}\!$ & $\!R\!$
			\Tstrut \\
			\hline
			\hline 
			$\!\!\!\!\! \theta \!=\! \frac{2}{3}$ & $20$     &  4.98e-5 &  -   & 2.57e-5  & -   &  1.48e-3  & -
			\Tstrut \\
			& $40$      &  6.08e-6   & 3.03   & 3.18e-6  & 3.02  & 3.78e-4  & 1.97
			\\
			& $60$      &  1.51e-6   & 3.43   & 8.03e-7  & 3.39  & 1.59e-4  & 2.14
			\\
			& $80$      &  5.74e-7   & 3.36   & 3.12e-7  & 3.29  & 8.81e-5  & 2.04
			\\
			& $100$    &  2.65e-7   & 3.46   & 1.46e-7  & 3.39  & 5.48e-5  & 2.13
			\\
			\hline
			\hline 
			$\!\!\!\!\! \theta \!=\! \frac{2}{\sqrt{5}}$ & $20$     &  4.97e-5 &  -   & 2.57e-5  & -   &  1.48e-3  & -
			\Tstrut \\
			& $40$      &  6.07e-6  &  3.03  &  3.17e-6  &  3.02  &  3.78e-4  &  1.97
			\\
			& $60$      &  1.51e-6  &  3.43  &  8.02e-7  &  3.39  &  1.59e-4  &  2.14
			\\
			& $80$      &  5.73e-7  &  3.37  &  3.11e-7  &  3.29  &  8.81e-5  &  2.04
			\\
			& $100$    &  2.63e-7  &  3.48  &  1.45e-7  &  3.41  &  5.47e-5  &  2.13
			\\
			\hline
			\hline 
			$\!\!\!\!\! \theta \!=\! 1$ & $20$     &  4.96e-5 &  -   & 2.56e-5  & -   &  1.48e-3  & -
			\Tstrut \\
			& $40$      &  6.06e-6  &  3.03  &  3.17e-6  &  3.02  &  3.78e-4  &  1.97
			\\
			& $60$      &  1.51e-6  &  3.43  &  8.01e-7  &  3.39  &  1.59e-4  &  2.14
			\\
			& $80$      &  5.72e-7  &  3.37  &  3.11e-7  &  3.29  &  8.81e-5  &  2.04
			\\
			& $100$    &  2.63e-7  &  3.49  &  1.45e-7  &  3.42  &  5.47e-5  &  2.13
			\\
			\hline
			\hline
		\end{tabular}
		\label{table:DLN-CSS-2D-h}
	\end{table}

    \begin{table}[ptbh]
		\centering
		\renewcommand\arraystretch{1.25}
		\caption{$L^{2}$-norm of error and rate for 2D DLN-SAV scheme in time ($N = 500$)} 
		\begin{tabular}{ccccccccc}
			\hline
			\hline
			& $\!k \!$  & $\!\| u \!-\! u^{h} \|_{\ell^{\infty}(L^{2})}\!$ & $\!R\!$
			& $\!\| u \!-\! u^{h} \|_{\ell^{2}(L^{2})}\!$ & $\!R\!$ 
			& $\!\| u \!-\! u^{h} \|_{\ell^{2}(H^{1})}\!$ & $\!R\!$
			\Tstrut
			\\
			\hline 
			\hline 
			$\!\!\!\!\!\theta \!=\! \frac{2}{3}$ & $0.4$     & 1.47e-3   & -   & 1.25e-3   &  - & 1.77e-3   & -
			\Tstrut \\
			& $0.2$      &  5.08e-4  & 1.53  & 3.92e-4   & 1.67  & 5.55e-4   & 1.67
			\\
			& $0.1$      & 1.45e-4   & 1.81   & 1.07e-4   & 1.88  & 1.51e-4   & 1.88
			\\
			& $0.05$    & 3.84e-5   & 1.92   & 2.75e-5   & 1.96  &  3.91e-5 &  1.95
			\\
			\hline
			\hline 
			$\!\!\!\!\!\theta \!=\! \frac{2}{\sqrt{5}}$ & $0.4$     & 2.07e-3   &  -  &  1.76e-3  & -  &   2.49e-3 & -
			\Tstrut \\
			& $0.2$      & 7.39e-4   & 1.49   &  5.71e-4   & 1.63   & 8.07e-4   & 1.63
			\\
			& $0.1$      & 2.15e-4   & 1.78   &  1.58e-4   & 1.86   & 2.23e-4   & 1.86
			\\
			& $0.05$    & 5.72e-5   & 1.91   &  4.09e-5   & 1.94   & 5.80e-5   & 1.94
			\\
			\hline
			\hline 
			$\!\!\!\!\!\theta \!=\! 1$ & $0.4$     & 2.36e-3   & -    & 2.01e-3  & -   & 2.84e-3   & -
			\Tstrut \\
			& $0.2$       & 8.53e-4   & 1.47   & 6.58e-4   & 1.61  &  9.31e-4   & 1.61
			\\
			& $0.1$       & 2.49e-4   & 1.77   & 1.83e-4   & 1.85  & 2.59e-4    & 1.85
			\\
			& $0.05$     & 6.67e-5   & 1.90   & 4.77e-5   & 1.94  & 6.76e-5    & 1.94
			\\
			\hline
			\hline
		\end{tabular}
		\label{table:DLN-SAV-2D-dt}
	\end{table}

    \begin{table}[ptbh]
		\centering
		\renewcommand\arraystretch{1.25}
		\caption{$L^{2}$-norm of error and rate for 2D DLN-SAV scheme in space ($k=0.01$)} 
		\begin{tabular}{ccccccccc}
			\hline
			\hline
			& $\!N\! $  & $\!\| u \!-\! u^{h} \|_{\ell^{\infty}(L^{2})}\!$ & $\!R\!$
			& $\!\| u \!-\! u^{h} \|_{\ell^{2}(L^{2})}\!$ & $\!R\!$ 
			& $\!\| u \!-\! u^{h} \|_{\ell^{2}(H^{1})}\!$ & $\!R\!$
			\Tstrut \\
			\hline
			\hline 
			$\!\!\!\!\! \theta \!=\! \frac{2}{3}$ & $20$     &  4.97e-5 &  -   &  2.57e-5  & -   &  1.48e-3  & -
			\Tstrut \\
			& $40$      &  6.08e-6  & 3.03  &  3.18e-6  & 3.01  &  3.78e-4  &  1.97
			\\
			& $60$      &  1.52e-6  & 3.43  &  8.07e-7  & 3.38  &  1.59e-4  &  2.14
			\\
			& $80$      &  5.80e-7  & 3.34  &  3.15e-7  & 3.27  &  8.81e-5  &  2.04
			\\
			& $100$    &  2.71e-7  & 3.41  &  1.49e-7  & 3.35  &  5.48e-5  &  2.13
			\\
			\hline
			\hline 
			$\!\!\!\!\! \theta \!=\! \frac{2}{\sqrt{5}}$ & $20$     &  4.97e-5 &  -   &  2.56e-5  & -   &  1.48e-3  & -
			\Tstrut \\
			& $40$      &  6.08e-6  &  3.03  &  3.18e-6  & 3.01  & 3.78e-4  & 1.97
			\\
			& $60$      &  1.52e-6  &  3.42  &  8.06e-7  & 3.38  & 1.59e-4  & 2.14
			\\
			& $80$      &  5.84e-7  &  3.32  &  3.16e-7  & 3.25  & 8.81e-5  & 2.04
			\\
			& $100$    &  2.78e-7  &  3.32  &  1.53e-7  & 3.27  & 5.48e-5  & 2.13
			\\
			\hline
			\hline 
			$\!\!\!\!\! \theta \!=\! 1$ & $20$     &  4.96e-5 &  -   &  2.56e-5  & -   &  1.48e-3  & -
			\Tstrut \\
			& $40$      &  6.07e-6  &  3.03  &  3.17e-6  &  3.01  &  3.78e-4  &  1.97
			\\
			& $60$      &  1.52e-6  &  3.42  &  8.06e-7  &  3.38  &  1.59e-4  &  2.14
			\\
			& $80$      &  5.86e-7  &  3.31  &  3.17e-7  &  3.24  &  8.81e-5  &  2.04
			\\
			& $100$    &  2.83e-7  &  3.26  &  1.55e-7  &  3.22  &  5.47e-5  &  2.13
			\\
			\hline
			\hline
		\end{tabular}
		\label{table:DLN-SAV-2D-h}
	\end{table}

    \subsubsection{Variable Time Step}
    Similar to Section 6.1.2, we test the accuracy of the variable time-stepping modified DLN algorithm \eqref{eq:DLN-CSS-Alg} and DLN-SAV algorithm \eqref{eq:DLN-SAV-Alg-weak} over time for the 2D case, utilizing the randomized time step $k_n = k + k \cdot \text{rand}$. The final time is $T = 4.0$ and the number of nodes $N=500$ on each side of the boundaries. The accuracy rate calculated using $k_{\text{max}}$ again demonstrates second-order temporal convergence rate numerically, as seen in Tables \ref{table:DLN-CSS-2D-random-t} and \ref{table:DLN-SAV-2D-random-t}. 

    \begin{table}[ptbh]
		\centering
		\renewcommand\arraystretch{1.25}
		\caption{$L^{2}$-norm of error and rate for 2D modified DLN scheme in time ($k_n = k + k \cdot rand; N = 500$)} 
		\begin{tabular}{ccccccccc}
			\hline
			\hline
			& $k$ & $\!k_{\text{max}}\! $  & $\!\| u \!-\! u^{h} \|_{\ell^{\infty}(L^{2})}\!$ & $\!R\!$
			& $\!\| u \!-\! u^{h} \|_{\ell^{2}(L^{2})}\!$ & $\!R\!$ 
			& $\!\| u \!-\! u^{h} \|_{\ell^{2}(H^{1})}\!$ & $\!R\!$
			\Tstrut
			\\
			\hline
			\hline
			$\!\!\!\!\!\theta \!=\! \frac{2}{3}$ & $0.4$ & 0.7818  & 1.67e-2   & - & 1.59e-2    & -         & 2.24e-2   & - 
			\Tstrut
			\\
                & $0.3$ & 0.5851  & 8.64e-3   & 2.27 & 7.39e-3    & 2.63    & 1.05e-2   & 2.63   
			\\
			& $0.2$ & 0.3683   & 3.58e-3   & 1.90 & 2.94e-3    & 1.99    & 4.16e-3   & 1.99  
			\\
			& $0.1$ & 0.1987  & 9.36e-4   & 2.18 & 7.19e-4    & 2.28    & 1.02e-3   & 2.28   
			\\
			\hline
			\hline
			$\!\!\!\!\! \theta \!=\! \frac{2}{\sqrt{5}}$ & $0.4$ & 0.7818  & 8.10e-3  & -  & 7.69e-3   & -         & 1.09e-2   & -        
			\Tstrut
                \\
			& $0.3$ & 0.5851 & 4.26e-3   & 2.22 & 3.64e-3   & 2.58    & 5.15e-3   & 2.58   
			\\
			& $0.2$ & 0.3683 & 1.77e-3   & 1.89 & 1.45e-3   & 1.99    & 2.05e-3   & 1.99   
			\\
			& $0.1$ & 0.1987  & 4.71e-4   & 2.15 & 3.59e-4  & 2.26    & 5.08e-4   & 2.26   
			\\
			\hline
			\hline
			$\!\!\!\!\! \theta \!=\! 1$ & $0.4$  & 0.7818  & 3.61e-3   & - & 3.41e-3   & -         & 4.83e-3   & -        
			\Tstrut
                \\
			& $0.3$ & 0.5851  & 1.93e-3   & 2.16 & 1.64e-3   & 2.54    & 2.32e-3   & 2.54  
			\\
			& $0.2$ & 0.3683  & 7.94e-4   & 1.92 & 6.46e-4   & 2.01    & 9.13e-4   & 2.01   
			\\
			& $0.1$ & 0.1987  & 2.14e-4   & 2.12 & 1.60e-4  & 2.26    & 2.27e-4  & 2.26   
			\\
			\hline
			\hline
		\end{tabular}
		\label{table:DLN-CSS-2D-random-t}
	\end{table}

    \begin{table}[ptbh]
		\centering
		\renewcommand\arraystretch{1.25}
		\caption{$L^{2}$-norm of error and rate for 2D DLN-SAV scheme in time ($k_n = k + k \cdot rand; N = 500$)} 
		\begin{tabular}{ccccccccc}
			\hline
			\hline
			& $k$ & $\!k_{\text{max}}\! $  & $\!\| u \!-\! u^{h} \|_{\ell^{\infty}(L^{2})}\!$ & $\!R\!$
			& $\!\| u \!-\! u^{h} \|_{\ell^{2}(L^{2})}\!$ & $\!R\!$ 
			& $\!\| u \!-\! u^{h} \|_{\ell^{2}(H^{1})}\!$ & $\!R\!$
			\Tstrut
			\\
			\hline
			\hline
			$\!\!\!\!\!\theta \!=\! \frac{2}{3}$ 
                & $0.4$ & 0.7818  & 3.37e-3   & -    & 3.27e-3    & -       & 4.63e-3   & - 
			\Tstrut
			\\
                & $0.3$ & 0.5851  & 2.17e-3   & 1.52 & 1.87e-3    & 1.94    & 2.64e-3   & 1.94  
			\\
			& $0.2$ & 0.3683   & 1.10e-3   & 1.47 & 9.02e-4   & 1.57    & 1.28e-3   & 1.57 
			\\
			& $0.1$ & 0.1987  & 3.10e-4   & 2.05 & 2.33e-4    & 2.20    & 3.29e-4   & 2.20   
			\\
			\hline
			\hline
			$\!\!\!\!\! \theta \!=\! \frac{2}{\sqrt{5}}$ 
                & $0.4$ & 0.7818  & 4.25e-3   & -    & 4.12e-3    & -       & 5.83e-3   & - 
			\Tstrut
			\\
                & $0.3$ & 0.5851  & 2.87e-3   & 1.35 & 2.47e-3    & 1.76    & 3.50e-3   & 1.76  
			\\
			& $0.2$ & 0.3683   & 1.52e-3   & 1.38 & 1.25e-3   & 1.48    & 1.76e-3   & 1.48 
			\\
			& $0.1$ & 0.1987  & 4.48e-4   & 1.98 & 3.36e-4    & 2.13    & 4.75e-4   & 2.13   
			\\
			\hline
			\hline
			$\!\!\!\!\! \theta \!=\! 1$ 
                & $0.4$ & 0.7818  & 4.65e-3   & -    & 4.51e-3    & -       & 6.37e-3   & - 
			\Tstrut
			\\
                & $0.3$ & 0.5851  & 3.21e-3   & 1.28 & 2.76e-3    & 1.69    & 3.90e-3   & 1.69
			\\
			& $0.2$ & 0.3683   & 1.72e-3   & 1.35 & 1.41e-3   & 1.45    & 2.00e-3   & 1.45
			\\
			& $0.1$ & 0.1987  & 5.16e-4   & 1.95 & 3.87e-4    & 2.10    & 5.47e-4   & 2.10  
			\\
			\hline
			\hline
		\end{tabular}
		\label{table:DLN-SAV-2D-random-t}
	\end{table}

    \subsubsection{Adaptive Algorithms}
    Next, we compare Algorithm \ref{alg:Adaptivity-AB2-like} and the corresponding constant time-stepping algorithms to show the accuracy and efficiency of time adaptivity. 
    e set final time $T = 1$, model parameter $\epsilon = 0.01$, the number of nodes $N = 50$ on each side of the boundaries for both adaptive and constant algorithms. 
    For adaptive algorithms, we set the maximum time step $k_{\rm{max}} = 0.1$,  the minimum time step $k_{\rm{min}} = 1.\rm{e}-5$, two initial time step $k_{0} = k_{1} = 1.\rm{e}-3$, tolerance $\rm{Tol} = 1.\rm{e}-8$ and safety factor $\kappa = 0.8$. We use time step $k = 1.\rm{e}-3$ for constant time-stepping algorithms. 
    From Table \ref{table:Adapt-Const-DLN-CSS-2D}, we can see that adaptive modified DLN algorithms, especially with $\theta = 2/3$ or $1$, work more efficiently than the corresponding constant time-stepping algorithms, since the adaptive algorithms take much fewer number of time steps and obtain errors comparable to those of the constant step algorithms. 
    For DLN-SAV algorithms, we observe from Table \ref{table:Adapt-Const-DLN-SAV-2D} that adaptive algorithms have the same error magnitude as constant time-stepping algorithms. However, adaptive algorithms with all three $\theta$ values finish the simulation only in $45$ time steps while the constant step algorithms take $1000$ time steps.

    \begin{table}[ptbh]
		\centering
		\renewcommand\arraystretch{1.25}
		\caption{Errors and number of time steps of the adaptive and constant time-stepping 2D modified DLN schemes} 
		\begin{tabular}{cccccc}
			\hline
			\hline
			\multicolumn{6}{c}{Adaptive modified DLN schemes}          
			\Tstrut \\
			\hline
			$\!\!\!\!\theta \!$  & $\!\| u \!-\! u^{h} \|_{\ell^{\infty}(L^{2})} \!$  & $\!\| u \!-\! u^{h} \|_{\ell^{2}(L^{2})}\!$ 
			& $\!\| u \!-\! u^{h} \|_{\ell^{2}(H^{1})}\!$ & \# Steps & \# Rejections 
			\Tstrut \\
			\hline 
			$\!\!\!\!\!\theta \!=\! \frac{2}{3}$                 & 3.02e-6    & 1.79e-6    & 2.37e-4    & 280    & 73    
			\\
			$\!\!\!\!\! \theta \!=\! \frac{2}{\sqrt{5}}$      & 4.86e-6    & 2.96e-6    & 2.38e-4    & 93      & 18    
			\\
			$\!\!\!\!\! \theta \!=\! 1$                              & 2.96e-6    & 1.60e-6    & 2.37e-4    & 48      & 18    
			\\
			\hline
			\hline
			\multicolumn{6}{c}{Constant time-stepping modified DLN schemes}          
			\Tstrut \\
			\hline
			$\!\!\!\!\theta \!$  & $\!\| u \!-\! u^{h} \|_{\ell^{\infty}(L^{2})} \!$  & $\!\| u \!-\! u^{h} \|_{\ell^{2}(L^{2})}\!$ 
			& $\!\| u \!-\! u^{h} \|_{\ell^{2}(H^{1})}\!$ & \# Steps & \# Rejections 
			\Tstrut \\
			\hline
			$\!\!\!\!\!\theta \!=\! \frac{2}{3}$                  &  2.97e-6    & 1.56e-6    & 2.37e-4     & 1000    & - 
			\\
			$\!\!\!\!\! \theta \!=\! \frac{2}{\sqrt{5}}$       &  2.97e-6    & 1.56e-6    & 2.37e-4     & 1000    & - 
			\\
			$\!\!\!\!\! \theta \!=\! 1$                               &  2.97e-6    & 1.56e-6    & 2.37e-4     & 1000    & - 
			\\
			\hline
		\end{tabular}
		\label{table:Adapt-Const-DLN-CSS-2D}
	\end{table}

    \begin{table}[ptbh]
		\centering
		\renewcommand\arraystretch{1.25}
		\caption{Errors and number of time steps of the adaptive and constant time-stepping 2D DLN-SAV schemes} 
		\begin{tabular}{cccccc}
			\hline
			\hline
			\multicolumn{6}{c}{Adaptive DLN-SAV schemes}          
			\Tstrut \\
			\hline
			$\!\!\!\!\theta \!$  & $\!\| u \!-\! u^{h} \|_{\ell^{\infty}(L^{2})} \!$  & $\!\| u \!-\! u^{h} \|_{\ell^{2}(L^{2})}\!$ 
			& $\!\| u \!-\! u^{h} \|_{\ell^{2}(H^{1})}\!$ & \# Steps & \# Rejections 
			\Tstrut
			\\
			\hline 
			$\!\!\!\!\!\theta \!=\! \frac{2}{3}$                 & 3.03e-6    & 1.62e-6    & 2.37e-4   & 45    & 9      
			\\
			$\!\!\!\!\! \theta \!=\! \frac{2}{\sqrt{5}}$      & 3.08e-6    & 1.65e-6    & 2.37e-4   & 44    & 9      
			\\
			$\!\!\!\!\! \theta \!=\! 1$                              & 3.11e-6    & 1.66e-6    & 2.37e-4   & 44    & 15    
			\\
			\hline
			\hline
			\multicolumn{6}{c}{Constant time-stepping DLN-SAV schemes}          
			\Tstrut \\
			\hline
			$\!\!\!\!\theta \!$  & $\!\| u \!-\! u^{h} \|_{\ell^{\infty}(L^{2})} \!$  & $\!\| u \!-\! u^{h} \|_{\ell^{2}(L^{2})}\!$ 
			& $\!\| u \!-\! u^{h} \|_{\ell^{2}(H^{1})}\!$ & \# Steps & \# Rejections 
			\Tstrut \\
			\hline
			$\!\!\!\!\!\theta \!=\! \frac{2}{3}$                  & 2.97e-6    & 1.56e-6    & 2.37e-4    & 1000    & -    
			\\
			$\!\!\!\!\! \theta \!=\! \frac{2}{\sqrt{5}}$       & 2.97e-6    & 1.56e-6    & 2.37e-4    & 1000    & -    
			\\
			$\!\!\!\!\! \theta \!=\! 1$                               & 2.97e-6    & 1.56e-6    & 2.37e-4    & 1000    & -    
			\\
			\hline
		\end{tabular}
		\label{table:Adapt-Const-DLN-SAV-2D}
	\end{table}

    \subsection{Two-dimensional test with random initial value}
	In this example, we simulate 2D Allen-Cahn equation with random initial condition \cite{HP19_NMPDE}: 
    \begin{gather*}
        u_0(x,y) = 0.1 \times \text{rand}(x,y) - 0.05, \quad   (x,y) \in [0,2\pi]^2,
    \end{gather*}
    where the function $\text{rand}(x,y)$ generates value uniformly distributed on $[0,1]$ at point $(x,y)$.
    We set the model parameter $\epsilon$ to be $0.1$ and number of nodes on each side of boundaries to be $100$ to generate a triangular mesh.
    We test the adaptive Algorithm \ref{alg:Adaptivity-AB2-like} with periodic boundary condition.

    For both adaptive modified DLN and DLN-SAV algorithms, we set two initial time step $k_{0}=k_{1} = 0.01$, the maximum time step $k_{\rm{max}} = 0.1$, the minimum step $k_{\rm{min}} = 1.\rm{e}-5$, the tolerance for LTE $\rm{Tol} = 1.\rm{e}-6$, and safety factor $\kappa = 0.8$. 
    We use the fixed point iteration with tolerance $1.\rm{e}-8$ for non-linear solver in the adaptive modified DLN algorithm.
    The results of three $\theta$ values ($2/3, 2/\sqrt{5},1$) are very close, thus we only present the numerical results for the case of $\theta = 1$. 
    From Figures \ref{fig:adaptCSSFPIrandperiodic} and \ref{fig:adaptSAVrandperiodic}, we observe that both modified DLN and DLN-SAV algorithms converge to the steady state at the time $T = 320$ with around $5800$ time steps. 
    We also test the corresponding constant time-stepping DLN algorithms with constant step $k = 0.01$ and observe that the constant time-stepping algorithms converge to the steady state with more than $30000$ time steps. 
    From this test problem, we demonstrate that the adaptive DLN algorithms are stable and possess superior time efficiency when compared to the corresponding constant time-stepping algorithms.

    \begin{figure}[ptbh]
		\begin{flushleft}
			\includegraphics[scale=0.172]{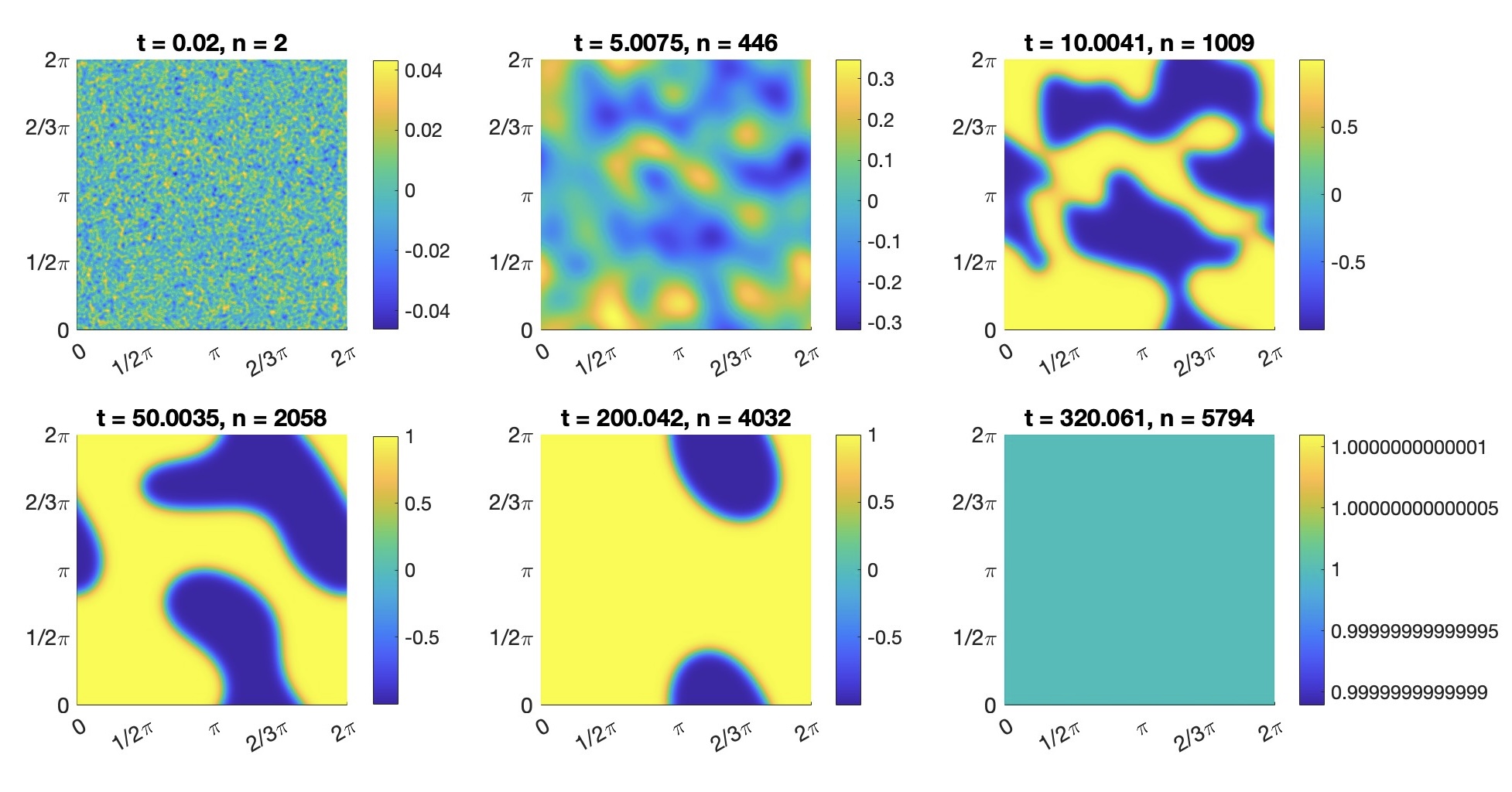}
			\caption{2D Adaptive modified DLN algorithm with $\theta= 1$ converges to steady state with 5794 time steps.}
			\label{fig:adaptCSSFPIrandperiodic}
		\end{flushleft}
	\end{figure}

    \begin{figure}[ptbh]
		\begin{flushleft}
			\includegraphics[scale=0.172]{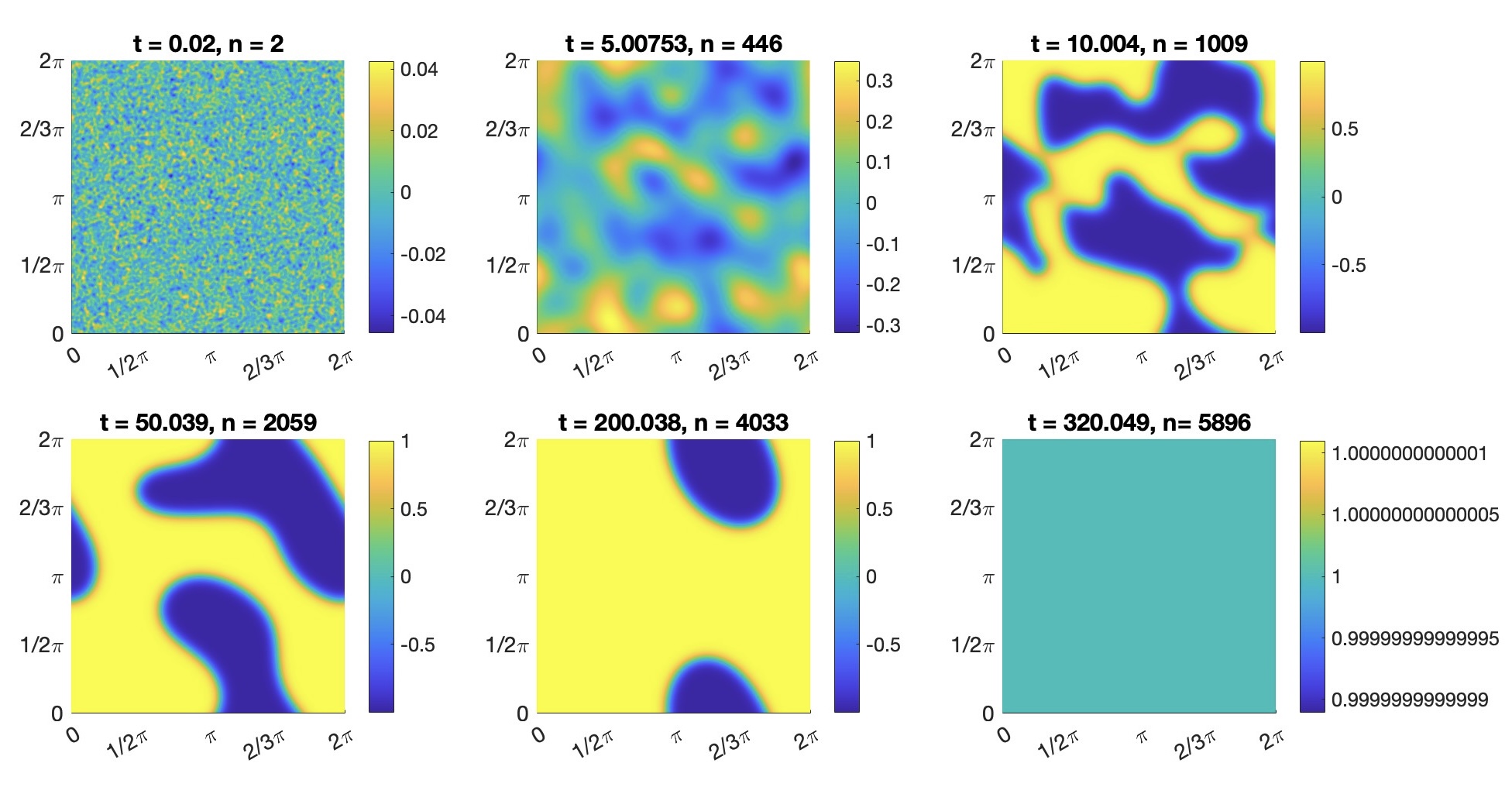}
			\caption{2D Adaptive DLN-SAV algorithm with $\theta= 1$ converges to steady state with 5896 time steps. }
			\label{fig:adaptSAVrandperiodic}
		\end{flushleft}
	\end{figure}

    \section{Conclusion} \label{sec7}
    In this paper, we propose two time-efficient algorithms for the Allen-Cahn equation: the modified DLN scheme and the DLN-SAV scheme. 
    We prove that the modified DLN scheme unconditionally satisfies the discrete energy dissipation law, and its numerical solutions under uniform time grids are second-order accurate in time. 
    When $\theta = 1$, the scheme reduces to the modified midpoint rule, and its numerical solutions are second-order in time on arbitrary time grids. 
    For the variable step DLN-SAV scheme, we show that the scheme approximately satisfies the discrete dissipation law. Moreover, its implementation can be simplified equivalently through a refactorization process on the BE-SAV algorithm. 
    To further enhance the performance of both algorithms, we utilized a time-adaptive mechanism based on the LTE criterion.

    We validate our methods through three main numerical tests. 
    The 1D traveling wave test and the 2D test with an extra source term show that both algorithms are second-order accurate under arbitrary time step sizes. 
    The 2D test with random initial conditions also verify that the modified DLN scheme and the DLN-SAV scheme are long-time stable under non-uniform time grids. 
    All three tests demonstrate that the time-adaptive versions of both algorithms outperform the corresponding constant step algorithms. 
    These adaptive algorithms achieve the same level of accuracy while requiring significantly fewer time steps.

    \begin{appendices}
        \section{Proof of Lemma \ref{lemma:consistency-1st}}
		\label{appendixA} \ \\
        For any $\theta \in [0,1)$, applying fundamental theorem of Calculus leads to
			\begin{align*}
				u_{n+1,\theta}(x) - u_{n,\theta}(x)
				&= \frac{1+\theta}{2} \big( u_{n+1}(x) - u_{n}(x) \big) 
				    + \frac{1-\theta}{2} \big( u_{n}(x) - u_{n-1}(x) \big) \\
				&= \frac{1+\theta}{2} \int_{t_{n}}^{t_{n+1}} u_{t}(x,t) dt 
				    + \frac{1-\theta}{2} \int_{t_{n-1}}^{t_{n}} u_{t}(x,t) dt.
			\end{align*}
		By H\"older's inequality, we have
			\begin{align*}
				\big\| (u_{n+1,\theta} - u_{n,\theta})^{2} \big\|^{2} 
				\leq& \Big( \frac{1+\theta}{2} \Big)^{4} \int_{\Omega} \Big( \int_{t_{n-1}}^{t_{n+1}} |u_{t}(x,t)| dt \Big)^{4} dx \\
				\leq& C(\theta) \int_{\Omega} \Big( \int_{t_{n-1}}^{t_{n+1}} 1 dt \Big)^{3}
				\Big( \int_{t_{n-1}}^{t_{n+1}} |u_{t}(x,t)|^{4} dt \Big) dx,
			\end{align*}
		which implies \eqref{eq:consist-1st-eq1}. 
		Similarly, 
			\begin{align*}
				&\frac{u_{n+1,\theta} + u_{n,\theta}}{2} - u(t_{n,\beta})
                    =\frac{u_{n+1,\theta} + u_{n,\theta}}{2} - u(t_n) +u(t_n) - u(t_{n,\beta})\\
				& \qquad = \frac{1 + \theta}{4} \int_{t_{n}}^{t_{n+1}} u_{t}(x,t) dt 
				- \frac{1 - \theta}{4} \int_{t_{n-1}}^{t_{n}} u_{t}(x,t) dt 
				- \int_{t_{n}}^{t_{n,\beta}} u_{t}(x,t) dt.
			\end{align*}
		Then Eq. \eqref{eq:consist-1st-eq2} follows by the fact $t_{n,\beta} \in [t_{n-1},t_{n+1}]$ and using H\"older's inequality:
			\begin{align*}
				\Big\| \frac{u_{n+1,\theta} + u_{n,\theta}}{2} - u(t_{n,\beta}) \Big\|^{2}
				\leq& \Big( \frac{1+\theta}{4} \Big)^{2} \int_{\Omega} \Big( \int_{t_{n-1}}^{t_{n+1}} |u_{t}(x,t)| dt \Big)^{2} dx \\
				\leq& C(\theta) \int_{\Omega} \Big( \int_{t_{n-1}}^{t_{n+1}} 1 dt \Big)
				\Big( \int_{t_{n-1}}^{t_{n+1}} |u_{t}(x,t)|^{2} dt \Big) dx.
			\end{align*}
        For the case of $\theta = 1$, the corresponding conclusions for the midpoint rule are easy to verify and the details are omitted here. 

        Next, we consider the case of uniform time grids with constant time step $k$ and aim to prove \eqref{eq:consist-1st-eq2-const}. Applying Taylor's theorem with integral remainder 
			\begin{align}
				u ( x, t_{n+1} ) &= u ( x, t_{n} ) + u_{t} ( x, t_{n} ) k 
				+ \int_{t_{n}}^{t_{n+1}} u_{tt} ( x, t ) ( t_{n+1} - t ) dt, 
				\label{eq:AppendixA-Int-remainder} \\
				u ( x, t_{n-1} ) &= u ( x, t_{n} ) - u_{t} ( x, t_{n} ) k
				+ \int_{t_{n}}^{t_{n-1}} u_{tt} ( x, t ) ( t_{n-1} - t ) dt, \notag \\
				u ( x, t_{n,\beta} ) &= u ( x, t_{n} ) + u_{t} ( x, t_{n} ) ( t_{n,\beta} - t_{n} ) 
				+ \int_{t_{n}}^{t_{n,\beta}} u_{tt} ( x, t ) ( t_{n,\beta} - t ) dt \notag \\
				&= u ( x, t_{n} ) + u_{t} ( x, t_{n} ) \big( {\beta_{2}^{(n)}}  - {\beta_{0}^{(n)}} \big) k
				+ \int_{t_{n}}^{t_{n,\beta}} u_{tt} ( x, t ) ( t_{n,\beta} - t ) dt, \notag 
			\end{align}
		with (under uniform time grids)
			\begin{gather*}
				\beta_{2}^{(n)} = \frac{1}{4} (2 + \theta - \theta^2 ), \qquad 
				\beta_{0}^{(n)} = \frac{1}{4} (2 - \theta - \theta^2 ),
			\end{gather*}
        gives
			\begin{align*}
				&\frac{u_{n+1,\theta} + u_{n,\theta}}{2} - u(x,t_{n,\beta}) \\
                    & = \frac{1 + \theta}{4} u(x,t_{n+1}) + \frac12 u(x,t_n) + \frac{1 - \theta}{4} u(x,t_{n-1}) - u(x,t_{n,\beta}) \\
				&= \frac{1 + \theta}{4} \int_{t_{n}}^{t_{n+1}} u_{tt} ( x, t ) ( t_{n+1} - t ) dt
				+ \frac{1 - \theta}{4} \int_{t_{n}}^{t_{n-1}} u_{tt} ( x, t ) ( t_{n-1} - t ) dt \\
				&\quad - \int_{t_{n}}^{t_{n,\beta}} u_{tt} ( x, t ) ( t_{n,\beta} - t ) dt.
			\end{align*}
		Therefore, Eq. \eqref{eq:consist-1st-eq2-const} follows from 
			\begin{align*}
				&\Big\|  \frac{u_{n+1,\theta} + u_{n,\theta}}{2} - u(t_{n,\beta}) \Big\|^{2} \\
				\leq& C(\theta)\!\!\! \int_{\Omega}\!\! \! \Bigg( \int_{t_{n}}^{t_{n\!+\!1}} \!\!\! |u_{tt} (x, t)|^{2} dt  \!\! \int_{t_{n}}^{t_{n\!+\!1}} \!\!\!(t_{n+1} \!-\! t)^{2} dt 
				\!\!+\!\! \int_{t_{n\!-\!1}}^{t_{n}}\!\!\! |u_{tt} (x,t) |^{2} dt 
				\!\! \int_{t_{n\!-\!1}}^{t_{n}}\!\!\! (t \!-\! t_{n\!-\!1})^{2} dt \notag \\
				&\qquad \qquad + \int_{t_{n}}^{t_{n,\beta}}\!\!\! | u_{tt} (x,t)|^{2} dt \!\! \int_{t_{n}}^{t_{n,\beta}} \!\! ( t_{n,\beta} \!-\! t )^{2} dt \Bigg) dx 
				\notag \\
				\leq& C(\theta) k^{3} \int_{\Omega} \int_{t_{n-1}}^{t_{n+1}}  |u_{tt} (x,t) |^{2} dt dx.
			\end{align*}

        \section{Proof of Lemma \ref{lemma:consistency-2nd}} 
        \label{appendixB} \ \\
        As in the proof of Lemma \ref{lemma:consistency-1st}, we only prove the case $\theta \in [0,1)$ below. The case for $\theta=1$ can be shown easily following the same procedure.
		Also, it suffices to consider the case $r = 0$ below. 
		Combining Taylor's theorem with integral remainder \eqref{eq:AppendixA-Int-remainder} (but with variable time step sizes $k_n$ and $k_{n-1}$)    
        with the fact that $\beta_{2}^{(n)} + \beta_{1}^{(n)} + \beta_{0}^{(n)} = 1$ leads to
			\begin{align}
				& u_{n,\beta}(x) - u(x,t_{n,\beta})   
				\label{eq:AppendB-eq1} \\
				=\, &{\beta_{2}^{(n)}}  \int_{t_{n}}^{t_{n+1}}  u_{tt} ( x, t ) ( t_{n\!+\!1} - t ) dt 
				+ {\beta_{0}^{(n)}}  \int_{t_{n}}^{t_{n-1}}  u_{tt} ( x, t ) ( t_{n\!-\!1} - t ) dt \notag \\
				&-  \int_{t_{n}}^{t_{n,\beta}}  u_{tt} ( x, t ) ( t_{n,\beta} - t ) dt. \notag 
			\end{align}
		Utilizing the H\"older's inequality, and bounding $\beta_{\ell}^{(n)}$ by $C(\theta)$, we have  
			\begin{align*}
				&\big\| u_{n,\beta}(\cdot) - u(\cdot,t_{n,\beta})  \big\|^{2} \notag \\
				\leq & \,C(\theta)\!\!\! \int_{\Omega} \Big[ \int_{t_{n}}^{t_{n+1}} | u_{tt} (x,t)| (t_{n\!+\!1} \!-\! t) dt 
				+ \int_{t_{n-1}}^{t_{n}} | u_{tt} (x,t)| (t \!-\! t_{n\!-\!1} ) dt \\
				& \quad \qquad + \int_{t_{n}}^{t_{n,\beta}} | u_{tt} (x,t) | |t_{n,\beta} \!-\! t| dt  \Big]^{2} dx 
				\notag \\
				\leq& \,C(\theta)\!\!\! \int_{\Omega}\!\! \Big\{\! \Big[ \!\! \int_{t_{n}}^{t_{n\!+\!1}} \!\!\! |u_{tt} (x, t)|^{2} dt  \!\!\! \int_{t_{n}}^{t_{n\!+\!1}} \!\!\!(t_{n+1} \!-\! t)^{2} dt \Big]^{\!\frac{1}{2}}
				\!\!+\!\! \Big[ \!\! \int_{t_{n\!-\!1}}^{t_{n}}\!\!\! |u_{tt} (x,t) |^{2} dt 
				\!\!\! \int_{t_{n\!-\!1}}^{t_{n}}\!\!\! (t \!-\! t_{n\!-\!1})^{2} dt \Big]^{\!\frac{1}{2}}  \notag \\
				&\qquad \qquad + \Big[ \!\!\int_{t_{n}}^{t_{n,\beta}}\!\!\! | u_{tt} (x,t)|^{2} dt \!\! \int_{t_{n}}^{t_{n,\beta}} \!\! ( t_{n,\beta} \!-\! t )^{2} dt \Big]^{\!\frac{1}{2}} \!\Big\}^{2} dx 
				\notag \\
				\leq& \,C(\theta) (k_{n} + k_{n-1})^{3} \!\int_{\Omega}\! \int_{t_{n-1}}^{t_{n+1}}  |u_{tt} (x,t) |^{2} dt dx,
				\notag 
			\end{align*}
		which implies \eqref{eq:consist-2nd-eq1} with $r=0$. 

        Next, we consider Eq. \eqref{eq:consist-2nd-eq2}.  We again consider Taylor's theorem with integral remainder
			\begin{align*}
				u (x,t_{n\!+\!1}) \!&=\! u(x,t_{n}) \!+\! u_{t}(x, t_{n}) {k_{n}} 
				\!+\! u_{tt}(x,t_{n}) \frac{k_{n}^{2}}{2} \!+\! \!\!\int_{t_{n}}^{t_{n\!+\!1}} \!\!u_{ttt}(x,t)
				\frac{(t_{n\!+\!1}\!-\!t)^{2}}{2} dt,  \\
				u(x,t_{n\!-\!1}) \!&=\! u(x,t_{n}) \!-\! u_{t}(x, t_{n}) {k_{n\!-\!1}} 
				\!+\! u_{tt}(x,t_{n}) \frac{k_{n\!-\!1}^{2}}{2} \!+\!\!\! \int_{t_{n}}^{t_{n\!-\!1}} \!\!u_{ttt} (x,t) 
				\frac{(t_{n\!-\!1} \!-\! t)^{2}}{2} dt,  \\
				u_{t}(x, t_{n,\beta}) \!&=\! u_{t} (x,t_{n}) \!+\! u_{tt} (x,t_{n}) ( {\beta_{2}^{(n)}} k_{n} \!-\! {\beta_{0}^{(n)}} k_{n\!-\!1} ) \!+\!\!\! \int_{t_{n}}^{t_{n,\beta}} \!\!u_{ttt} (x,t) (t_{n,\beta} \!-\! t) dt.  
			\end{align*}
		and combine them with the fact that $\alpha_{2} + \alpha_{1} + \alpha_{0} = 0$ to obtain
			\begin{align*}
				&\frac{u_{n,\alpha}(x)}{\widehat{k}_{n}} - u_{t}(t_{n,\beta}) 
				\\
				=& \Big[ \frac{{\alpha_{2}} {k_{n}}^{2} \!+\! {\alpha_{0}} {k_{n-1}}^{2}}{2 \widehat{k}_{n}} \!-\! \big( {\beta_{2}^{(n)}} k_{n} \!-\! {\beta_{0}^{(n)}} k_{n-1} \big) \Big] u_{tt}(x,t_{n}) \!-\!\! \int_{t_{n}}^{t_{n,\beta}} \!u_{ttt} (x,t) (t_{n,\beta} \!-\! t) dt  \\
				&+ \frac{{\alpha_{2}}}{ \widehat{k}_{n}} \!\!\int_{t_{n}}^{t_{n\!+\!1}} \!\!u_{ttt} (x,t) (t_{n\!+\!1} \!-\! t)^{2} dt 
				+ \frac{{\alpha_{0}}}{ \widehat{k}_{n}} \!\!\int_{t_{n}}^{t_{n\!-\!1}} \!\!u_{ttt} (x,t) (t_{n\!-\!1} \!-\! t)^{2} dt. 
			\end{align*}
		It's easy to check
			\begin{gather}
				\frac{{\alpha_{2}} {k_{n}}^{2} + {\alpha_{0}} {k_{n-1}}^{2}}{2 \widehat{k}_{n}} - \big( {\beta_{2}^{(n)}} k_{n} - {\beta_{0}^{(n)}} k_{n-1} \big) = 0.
				\label{eq:AppendB-eq3}
			\end{gather}
		Therefore,
			\begin{align*}
				&\Big\| \frac{u_{n,\alpha}(\cdot)}{\widehat{k}_{n}} - u_{t} (x,t_{n,\beta}) \Big\|^{2}  \\
				\leq& \frac{C(\theta)}{{\widehat{k}_{n}}^{2}}\!\! \int_{\Omega} \Big[ \int_{t_{n}}^{t_{n+1}} | u_{ttt} (x,t) | (t_{n+1} - t)^{2} dt + \int_{t_{n-1}}^{t_{n}} | u_{ttt} (x,t) | (t_{n-1} - t)^{2} dt  \\
				&\qquad \qquad + \widehat{k}_{n}  \int_{t_{n}}^{t_{n,\beta}} | u_{ttt} (x,t) | | t_{n,\beta} - t | dt \ \Big]^{2} dx  \\
				\leq& \!\!\frac{C(\theta)}{\widehat{k}_{n}^{2}} \!\!\! \int_{\Omega}\!\!\! \Big\{\! \Big[ \!\! \int_{t_{n}}^{t_{n\!+\!1}} \!\!\! |u_{ttt} (x,t)|^{2} dt  \!\!\! \int_{t_{n}}^{t_{n\!+\!1}} \!\!\!(t_{n+1} \!-\! t)^{4} dt \Big]^{\!\frac{1}{2}}
				\!\!+\!\! \Big[ \!\! \int_{t_{n\!-\!1}}^{t_{n}}\!\!\! |u_{ttt} (x,t)|^{2} dt 
				\!\!\! \int_{t_{n\!-\!1}}^{t_{n}}\!\!\! (t \!-\! t_{n\!-\!1})^{4} dt \Big]^{\!\frac{1}{2}}   \\
				&\qquad \qquad + \widehat{k}_{n} \Big[ \!\!\int_{t_{n}}^{t_{n,\beta}}\!\!\! | u_{ttt} (x,t)|^{2} dt \!\!\! \int_{t_{n}}^{t_{n,\beta}} \!\! (t_{n,\beta} \!-\! t)^{2} dt \Big]^{\!\frac{1}{2}} \!\Big\}^{2} dx 
				 \\
				\leq& C(\theta) \frac{(k_{n} + k_{n-1})^2+\widehat{k}_{n}^2}{\widehat{k}_{n}^2} (k_{n} + k_{n-1})^{3} \int_{\Omega} \int_{t_{n-1}}^{t_{n+1}} |u_{ttt} (x,t)|^{2} dt dx    \\
				\leq& C(\theta) (k_{n} + k_{n-1})^{3} \int_{\Omega} \int_{t_{n-1}}^{t_{n+1}} |u_{ttt} (x,t)|^{2} dt dx, 
			\end{align*}
        where the last inequality follows from $\widehat{k}_{n}=\frac{1+\theta}{2}k_{n} + \frac{1-\theta}{2}k_{n-1}$. The case with $r\ge 0$ can be proven in the same way.

    \end{appendices}

    \section*{Acknowledgements} 
    This work was initiated during D. Luo's participation in the 2023 Research Opportunities in Mathematics for Underrepresented Students (ROMUS) program at the Ohio State University (OSU) Department of Mathematics. The authors gratefully acknowledge the support provided by OSU and the NSF grant DMS-1753581.   
    
    \section*{Funding information}
    Y. Xing is partially supported by the National Science Foundation (NSF) grants DMS-1753581 and DMS-2309590.

    \section*{Author contributions} 
    The first two authors Y. Chen and D. Luo contributed equally to this work.

    \bibliographystyle{abbrv}

\end{document}